\documentclass{amsart}
\usepackage[latin1]{inputenc}
\usepackage{amssymb}

\newtheorem{thm}{Theorem}

\newtheorem{lem}[thm]{Lemma}
\newtheorem{prop}[thm]{Proposition}

\theoremstyle{definition}
\newtheorem{defn}[thm]{Definition}
\theoremstyle{remark}

\numberwithin{equation}{section}

\newcommand{\To}{\longrightarrow}

\begin{document}
\setcounter{tocdepth}{1}


\title[]{Compact spaces of the first Baire class that have open finite degree}
\author{Antonio Avil\'es and Stevo Todorcevic}
\address{Universidad de Murcia, Departamento de Matem\'{a}ticas, Campus de Espinardo 30100 Murcia, Spain.} \email{avileslo@um.es}
\address{Institut de Math\'{e}matiques de Jussieu, CNRS UMR 7586, Case 247, 4 place Jussieu, 75252 Paris Cedex, France. Department of Mathematics, University of Toronto, Toronto, Canada, M5S 3G3.}%
\email{stevo@math.jussieu.fr, stevo@math.toronto.edu}
\thanks{First author supported by MINECO and FEDER (MTM2014-54182-P) and by Fundaci\'{o}n S\'{e}neca - Regi\'{o}n de Murcia (19275/PI/14). Second author partially supported by grants from NSERC and CNRS}

\begin{abstract}
We introduce the open degree of a compact space, and we show that for every natural number $n$, the separable Rosenthal compact spaces of degree $n$ have a finite basis.
\end{abstract}

\maketitle

\section{Introduction}

A compact space is called a Rosenthal compactum if it is homeomorphic to a pointwise compact set of first Baire class functions on a Polish space. The study of this class arose in connexion with Banach space theory, and particularly with Banach spaces not containing $\ell_1$ \cite{OdeRos}. We refer to the survey~\cite{Debs} for more information on the subject. 

\begin{thm}\label{basisof3}
There exist three separable non-metrizable Rosenthal compact spaces such that every separable non-metrizable Rosenthal compact space contains a homeomorphic copy of one of these three.
\end{thm}

The three critical examples identified in \cite{1BC} are the split interval (also known as double arrow space), the Alexandroff duplicate of the Cantor set and the one-point compactification of the discrete set of size continuum. The two latter spaces are not separable, but there is a natural way to supplement them with countably many isolated points to make them separable Rosenthal compacta, and in this way we obtain the basis of three spaces of Theorem~\ref{basisof3}. Although not explictly mentioned in \cite{1BC}, these separable versions arise in the proofs. We can refer also to~\cite{ADK}, where the authors approach this topic emphasizing the role of dense sets indicated in the dyadic tree, and the basis of separable spaces is explicitly described. The main result of this paper is Theorem~\ref{finitebasis}, a multidimensional generalization of Theorem~\ref{basisof3}, that relies on the following topological index:

\begin{defn}\label{defodeg}
For a compact space $K$, the open degree of $K$, $odeg(K)$, is the least natural number $n$ such that there exists a countable family $\mathcal{F}$ of open sets such that for every different $x_0,\ldots,x_n\in K$ there exists $V_0,\ldots,V_n\in \mathcal{F}$ such that $x_i\in V_i$ and $\bigcap_i V_i = \emptyset$. If no such number $n$ exists, then $odeg(K) = \infty$.
\end{defn}

\begin{thm}\label{finitebasis}
For every natural number $n\geq 2$, there exists a finite list of separable Rosenthal compact spaces of open degree $n$ such that every separable Rosenthal compact space of open degree $\geq n$ contains a homeomorphic copy of one of the spaces of the list.
\end{thm}

It is an easy excercise that $odeg(K)= 1$ if and only if $K$ is metrizable, so for $n=2$, the list of Theorem~\ref{finitebasis} is the same as in Theorem~\ref{basisof3}. For higher numbers $n$, the list of compact spaces can also be explicitly described, and it is related to the list of minimal analytic strong $n$-gaps found in \cite{stronggaps}. Indeed, each compact space in the list of Theorem~\ref{finitebasis} corresponds to a \emph{dense} strong minimal analytic $n$-gap modulo permutations. These compact spaces are higher-dimensional analogues of the three critical spaces of Theorem~\ref{basisof3}, constructed using general $m$-adic trees $m^{<\omega}$ instead of  the dyadic tree $2^{<\omega}$. Our proof uses some of the ideas from \cite{1BC} but it also introduces some new techniques in the area, like the use of infinite games.

The paper is organized as follows: Section~\ref{sectionbiseq} collects well known or more or less elementary facts on Rosenthal compacta and the property of bisequentiality.  Section~\ref{sectiontree} introduces notation and facts that we shall need about the $n$-adic tree, that are taken from our previous work \cite{stronggaps,IHES,serbiansurvey}.  Sections~\ref{sectionK1} and \ref{sectionKinfty} describe the basic compact spaces $K_1(\mathfrak{P})$ and $K_\infty(\mathfrak{Q})$. In Section~\ref{sectionmin} we check when a space $K_1(\mathfrak{P})$ can be homeomorphic to a subspace of a space $K_1(\mathfrak{P}')$ and, using results from \cite{stronggaps}, the minimal spaces of this form are identified. 

Section~\ref{sectionopen} studies the open degree, and another related degree. It is proven in \cite{APT} that, for this other notion of degree, every separable Rosenthal compact space of degree $n$ contains one of two basic spaces. The differences with the results of this paper are discussed. Theorem~\ref{finitebasis} is finally proven in Section~\ref{sectionmain}, and Section~\ref{sectionlow} computes the minimal spaces for $n=2,3,5$ as illustration . Section~\ref{sectionclass} contains some further results like:

\begin{thm}
If $K$ is a non scattered Rosenthal compact space, then either $K$ contains a homeomorphic copy of the Cantor set or a homeomorphic copy of the split interval.
\end{thm}

\begin{thm}
If a Rosenthal compactum maps continuously onto the split interval, then it contains a copy of the split interval.
\end{thm}

Some problems are proposed in Section~\ref{sectionprob}.

\section{Rosenthal and bisequential compact spaces}\label{sectionbiseq}

Rosenthal showed that Rosenthal compact spaces are sequentially compact and have countable tightness \cite{Rosenthal}. This was improved by Bourgain, Fremlin and Talagrand \cite{BFT} that showed that these are Fr\'{e}chet-Urysohn spaces (every point in the closure of a set $A$ is the limit of a sequence in $A$). This was further improved by Pol to the property of bisequentiality \cite{Pol}. If $K$ is a topological space, a sequence $\{A_n : n<\omega\}$ of subsets of $K$ is said to converge to a point $x\in K$ if for every neighborhood $W$ of $x$ there exists $n_0$ such that $A_n\subset W$ for all $n>n_0$.

\begin{defn}
A compact space $K$ is bisequential if for every ultrafilter $\mathcal{U}$ on $K$ that converges to a point $x\in K$ there exists a sequence of sets $A_1,A_2,\ldots\in\mathcal{U}$ that converges to $x$.
\end{defn}

\begin{thm}[Pol] Every separable Rosenthal compactum is bisequential.
\end{thm}

The following fact noticed in \cite{Knaust} for Rosenthal compacta, can be obtained as an application of bisequentiality:

\begin{lem}\label{biseqblocks}
Let $K$ be a bisenquential compact space. Suppose that we have a sequence $\{x_k\}_{k<\omega}\subset K$ that clusters at $x$ and, for every $k$, a sequence $A_k = \{x_{kp}\}_{p<\omega}\subset K$ that clusters at $x_k$. Then there exist infinite sets $N,N_k\subset \omega$ for each $k$, such that the sets $\{\{x_{kp} : p\in N_k\} : k\in N\}$ converge to $x$.
\end{lem}

\begin{proof}
Let $$\mathcal{F} = \left\{ Z\subset K : |\{k<\omega : |\{p\in\omega : x_{kp}\not\in Z|=\omega\}|<\omega  \right\}.$$
This is a filter that clusters at the point $x$. By compactness, there exists an ultrafilter $\mathcal{U}$ that contains $\mathcal{F}$ and converges to $x$. Using bisequentiality we find a sequence $\{Z_n : n<\omega\}$ of elements of $\mathcal{U}$ that converges to $x$. The sets $K\setminus Z_n$ do not belong to $\mathcal{F}$, hence
$$M_n := |\{k<\omega : |\{p\in\omega : x_{kp}\in Z_n\}| = \omega\}|=\omega.$$
Choose inductively $k_1<k_2<k_3<\cdots$ such that $k_n\in M_n$ for every $n$. We take $N = \{k_1,k_2,\ldots\}$ and $N_{k_n} = \{p\in\omega : x_{k_n p}\in Z_n\}$.
\end{proof}

Notice that the above implies in particular that every bisequential space is Fr\'{e}chet-Urysohn, and therefore sequentially compact as well. The topology in a bisequential space is strongly determined by the behavior of convergent sequences. In particular, we have the following useful fact, that we shall frequently need:

\begin{lem}\label{continuousondense}
Let $K$ and $\tilde{K}$ be bisequential spaces, $D$ a dense subset $K$, and $\phi:D\To \tilde{K}$ a function. Then,
\begin{enumerate}
\item either there exists a continuous function $\Phi:K\To\tilde{K}$ that extends $\phi$,
\item or there exist two sequences $\{d_n\}$ and $\{d'_n\}$ in $D$ that converge to the same point, such that $\{\phi(d_n)\}$ and $\{\phi(d'_n)\}$ converge to two different points of $\tilde{K}$.
\end{enumerate}
\end{lem}

\begin{proof}
Suppose that (2) does not hold. For every $x\in K$ choose a sequence $(d_k)\subset D$ that converges to $x$ and such that $(\phi(d_k))$ is convergent, and then define $\Phi(x) = \lim_k \phi(d_k)$. Since bisequential spaces are Frechet-Urysohn and sequentially compact, such a sequence $d_k$ always exist. Moreover, by the failure of (2), the value $\Phi(x)$ does not depend on the choice of the sequence $(d_k)$. It remains to show that $\Phi$ is continuous. It is enough to check that it preserves limits of convergent sequences. So pick $(x_k)$ that converges to $x$ in $K$, $\Phi(x_k)$ converges to some $\tilde{x}\in \tilde{K}$, and we want to check hat $\tilde{x} = \Phi(x)$.  For every $k$ let $(d_{kp})_p$ be a sequence that converges to $x_k$ and $\phi(d_{kp})_p$ converges to $\Phi(x_k)$. Using Lemma~\ref{biseqblocks} twice, once in $K$ and once in $\tilde{K}$ we get a sequence of the form $(d_{kp[k]})$ that converges to $x$, and such that 
$(\phi(d_{kp[k]})$ converges to $\tilde{x}$. This implies that $\Phi(x) = \tilde{x}$.
\end{proof}

Here is another application of bisequentiality concerning scattered spaces. $K^{(\alpha)}$ denothes the $\alpha$-th Cantor-Bendixson derivative of $K$.

\begin{thm}
Let $K$ be a scattered bisequential compactum and $\alpha<\omega_1$ such that $K^{(\alpha)} \neq \emptyset$ . Then there exists a countable closed subset $L\subset K$ such that $L^{(\alpha)}\neq \emptyset$.
\end{thm}

\begin{proof}
By picking a neighborhood of a point in the $\alpha$-th level that isolates it in its level, we can suppose that $K^{(\alpha)} = \{\infty\}$ is a singleton. We shall proceed by induction on $\alpha$, the initial case being trivial. We consider a sequence $\{x_n\}\subset K$ of different points belonging to $K^{(\beta_n)}\setminus K^{(\beta_n+1)}$ where $\beta_n=\beta$ if $\alpha=\beta+1$ is successor, or $\lim_n\beta_n = \alpha$ if $\alpha$ is limit. By the inductive hypothesis, we can find countable compact sets $L_n$ of height $\beta_n$ with $L_n^{(\beta_n)} = \{x_n\}$. Let $L = \bigcup_n L_n$. We can consider $\mathcal{F}_0$ a filter on $K$ generated by all subsets $A\subset L$ for which $\overline{L_n\setminus A}^{(\beta_n)}=\emptyset$ for all but finitely many $n$'s. It is clear that $\mathcal{F}_0$ clusters at $\infty$. Let $\mathcal{F}$ be an ultrafilter that contains this filter and converges to $\infty$. By bisequentiality we can find a sequence of sets $L\supset F_1\supset F_2\supset \cdots$ in $\mathcal{F}$ that converges to $\infty$. Each $F_i$ has the property that $\overline{F_i\cap L_n}^{(\beta_n)} \neq\emptyset$ for infinitely many $n$'s, or otherwise $L\setminus F_i\in\mathcal{F}_0$ that would contradict that $F_i\in \mathcal{F}$. Inductively, choose $n[1]<n[2]<\cdots$ such that $\overline{F_{i}\cap L_{n[i]}}^{(\beta_{n[i]})} \neq\emptyset$. The closed set $\tilde{L} = \{\infty\}\cup \bigcup_i\overline{F_{i}\cap L_{n[i]}}$ is the one we are looking for.
\end{proof}

Finally, another well known fact on Rosenthal compacta.

\begin{lem}\label{densesetofcontinuous}
Let $K$ be a separable Rosenthal compact space, and $D\subset K$ a countable dense subset of $K$. Then there exists a countable familiy $\tilde{D}$ of continuous functions on $\omega^\omega$ whose pointwise closure consits of first Baire class functions on $\omega^\omega$, and a bijection $\phi:D\To \tilde{D}$ that extends to a homeomorphism between the closures.
\end{lem}

\begin{proof}
Suppose that $K$ is a compact set of first Baire class functions on the Polish space $X$. There is a finer Polish topology on $X$ for which all the functions in $D$ are continuous, cf. \cite[13A]{Kechris}. Find a continuous surjection $\pi:\omega^\omega\To X$ from $\omega^\omega$ onto $X$ with this new Polish topology. Just take $\tilde{D} = \{f\circ \pi : f\in D\}$ and $\phi(f) = f\circ \pi$. 
\end{proof}

\section{The $m$-adic tree}\label{sectiontree}

We consider a natural number as an ordinal, that equals its set of predecessors $m = \{0,1,\ldots,m-1\}$. Thus, $m\times m = \{0,1,\ldots,m-1\}\times \{0,1,\ldots,m-1\}$ denotes the cartesian product.

The $m$-adic tree is the set $m^{<\omega}$ of finite sequences of numbers from $m = \{0,1,\ldots,m-1\}$. The set $m^{\omega}$ is the set of infinite sequences of numbers from $m$, and $m^{\leq \omega} = m^{<\omega}\cup m^{\omega}$. The length of $s\in m^{\leq \omega}$, denoted as $|s|$, is the cardinality of its domain. Given $s=(s_0,\ldots,s_p)\in m^{<\omega}$ and $t=(t_0,t_1,\ldots)$ in $m^{\leq\omega}$, we can construct the concatenation $s^\frown t = (s_0,\ldots,s_p,t_0,t_1,\ldots)$. We write $s\leq t$ if there exists $r\in m^{\leq\omega}$ such that $t = s^\frown r$. When $r=(i)$ is a sequence with a single number, we write $s^\frown i$ instead of $s^\frown (i)$ for short. For any $s,t\in m^{\leq \omega}$, we denote by $s\wedge t$ the largest $r$ such that $r\leq t$ and $r\leq s$. If $r = s\wedge t \not\in \{s,t\}$ then there exist two different $i,j\in m$ such that $r^\frown i \leq s$ and $r^\frown j \leq t$. In that case, we define the incidence as $inc(s,t) = (i,j)$. If we have $i\in m$ such that $s^\frown i \leq t$, then we define $inc(t,s) = (i,i)$. The well order $\prec$ on $n^{<\omega}$ is given by $s\prec t$ if either $|s|<|t|$, or $|s|=|t| = p$ and $n^{p} s_0 + n^{p-1}s_1 + \ldots < n^p t_0 + n^{p-1}t_1 + \ldots$.

We are going to follow the approach in \cite{serbiansurvey} to study the \emph{first-move combinatorics} of $m^{<\omega}$. A set $A\subset m^{<\omega}$ will be said to be meet-closed if $t\wedge s\in A$ whenever $t\in A$ and $s\in A$. The meet-closure of $A$ is $\langle\langle A\rangle\rangle = \{s\wedge t : s,t\in A\}$ is the least meet-closed set that contains $A$.\\

For $A,B\subset m^{<\omega}$, a bijection $f:A\To B$ is a first-move-equivalence if it is the restriction of a bijection $f:\langle\langle A\rangle\rangle\To \langle\langle B\rangle\rangle$ such that for every $t,s\in \langle\langle A\rangle\rangle$
\begin{enumerate}
\item $f(t\wedge s) = f(t)\wedge f(s)$
\item $f(t) \prec f(s)$ if and only if $t\prec s$
\item If $i\in n$ is such that $t^\frown i \leq s$, then $f(t)^\frown i \leq f(s)$.
\end{enumerate}

For the third condition, we can also write that $inc(s,t) = inc(f(s),f(t))$ for all $s,t$. The sets $A$ and $B$ are called first-move-equivalent if there is a first-move-equivalence between them. In this case, we write $A\approx B$. There are several equivalence classes for this relation that we will use at some point:

A set $A\subset m^{<\omega}$ is a first-move subtree if $A\approx m^{<\omega}$.

An $(i,j)$ comb is a subset $A\subset m^{<\omega}$ such that
$$A\approx \{(j), (iij), (iiiij), (iiiiiij),\ldots\}.$$
An $(i,j,k,l)$-double comb is a subset $A\subset m^{<\omega}$ such that
$$A \approx \{(j), (iikkj), (iikkiikkj),\ldots\}\cup \{(iil),(iikkiil),(iikkiikkiil),\ldots\}$$
Notice that an $(i,j,k,l)$-comb is a union of an $(i,j)$-comb and a $(k,l)$-comb lying on the same branch.
An $(u,v)$-splitted $(i,j,k,l)$-double comb is a subset $A\subset m^{<\omega}$ such that
$$ A \approx \{(uj), (uiij), (uiiiij),\ldots\}\cup \{(vl), (vkkl), (vkkkkl),\ldots\}.$$
This is now the union of an $(i,j)$-comb and a $(k,l)$-comb lying on different branches. Let us also say that a function $f:A\To B$ is a first-move embedding if it is a first-move equivalence between $A$ and the range $f(A)$.

The following theorem is related to Milliken's partition theorem for trees\cite{Milliken}. A proof of it can be found in \cite{serbiansurvey} based on \cite{IHES}:

\begin{thm}\label{strongRamsey}
Fix a set $A_0\subset n^{<\omega}$, and a partition $\{A\subset n^{<\omega} : A\approx A_0\} = P_1\cup\cdots\cup P_k$ into finitely many sets with the property of Baire. Then there exists a first-move subtree $T\subset n^{<\omega}$ such that the family $\{A\subset T : A\approx A_0\}$ is contained in a single piece of the partion.
\end{thm}

Here, the property of Baire means that the sets $P_i$ belong to the $\sigma$-algebra generated by the open sets and the meager sets of $2^{n^{<\omega}}$, endowed with the product topology that makes it homeomorphic to the Cantor set.

Let $i,j\in m$, $x\in m^\omega$, and $\{s^0,s^1,\ldots\}\subset m^{\leq\omega}$. We say that $\{s^0,s^1,\ldots\}$ is an $(i,j)$-sequence over $x$ if the two following conditions hold:
\begin{enumerate}
\item $\lim_{n\rightarrow\infty} |s^n \wedge x| = \infty$,
\item $inc(x,x_n) = (i,j)$ for all $n$. 
\end{enumerate}

If $\{s^0,s^1,\ldots\}\subset m^\omega$ is an $(i,j)$-sequence over $x$, then it converges to $x$ in the natural product topology of $m^\omega$. Notice that an $(i,j)$-comb is always an $(i,j)$-sequence over some $x\in m^\omega$.

\begin{lem}[Lemma 7 of \cite{stronggaps}]\label{densecombs} Every infinite subset of $m^{<\omega}$ contains an infinite $(i,j)$-comb for some $(i,j)\in m\times m$.
\end{lem}

The families of sets $\mathcal{I}_1,\ldots,\mathcal{I}_{n}$ are said to be countably separated if there exists a countable family of sets $\mathcal{C}$ such that for every $a_1\in \mathcal{I}_,\ldots,a_n\in \mathcal{I}_n$ there exist $b_1,\ldots,b_n\in \mathcal{C}$ such that $a_i\setminus b_i$ is finite for all $i$, and $\bigcap_i b_i=\emptyset$.

\begin{lem}[Proposition 6 of \cite{stronggaps}]\label{notcountablyseparated}
Let $\mathfrak{P}$ be a partition of $m\times m$. For every $P\in\mathfrak{P}$, let $\mathcal{I}_P$ be the family of all $(i,j)$-combs of $m^{<\omega}$ such that $(i,j)\in P$. The families $\{\mathcal{I}_P : P\in\mathfrak{P}\}$ are not countably separated.

\end{lem}

\section{The spaces $K_1(\mathfrak{P})$}\label{sectionK1}

Given a partition $\mathfrak{P}$ of $m\times m$, we will construct a first countable Rosenthal compact space $K_1(\mathfrak{P})$. For this, we consider the Polish space $X_{\mathfrak{P}} := m^{<\omega}\cup m^{\omega} \times \mathfrak{P}$, where

\begin{itemize}
\item $m^{<\omega}$ is considered as a countable discrete space,
\item $m^{\omega}$ is considered with its product topology as the countable power of a finite discrete space, which makes it homeomorphic to the Cantor set,
\item $\mathfrak{P}$ is considered as a finite discrete space,
\item $ m^{\omega} \times \mathfrak{P}$ is given the product topology,
\item  $X_{\mathfrak{P}} = m^{<\omega}\cup m^{\omega} \times \mathfrak{P}$ is given the topology where each of the two sets $m^{<\omega}$ and $ m^{\omega} \times \mathfrak{P}$ are open, and endowed with the respective topologies indicated above.
\end{itemize}

Now, for every $s\in m^{<\omega}$ we consider a function $\mathbf{f}_s: X_\mathfrak{P}\To \{0,1\}$ given by
$$\mathbf{f}_s(t) =  \begin{cases} 1 \text{ if } t\leq s \\ 0 \text{ otherwise } \end{cases} \text{ for } s,t\in m^{<\omega},$$
$$\mathbf{f}_{s}(y,Q) = \begin{cases}
 1 \text{ if }  inc(y,s) \in Q\\
 0 \text{ if } inc(y,s) \not\in Q
   \end{cases} \text{ for } s\in m^{<\omega}, \ (y,Q)\in m^\omega\times \mathfrak{P}$$
   
\begin{defn}
The compact space $K_1(\mathfrak{P})$ is the pointwise closure of the functions $\{\mathbf{f}_s : s\in m^{<\omega}\}$ in $\mathbb{R}^{X_{\mathfrak{P}}}$.
\end{defn}

In order to describe all the points of $K_1(\mathfrak{P})$, for every $(x,P)\in m^\omega \times \mathfrak{P}$, we consider a new function $\mathbf{f}_{(x,P)}:X_\mathfrak{P} \To \{0,1\}$ given by

$$\mathbf{f}_{(x,P)}(t) = \begin{cases} 1 \text{ if } t\leq x \\ 0 \text{ otherwise } \end{cases} \text{ for } (x,P)\in m^\omega\times \mathfrak{P}, t\in m^{<\omega}$$
$$\mathbf{f}_{(x,P)}(y,Q) = \begin{cases}
 1 \text{ if } x=y, \ P = Q \\
 0 \text{ if } x=y, \ P\neq Q\\
 1 \text{ if } x\neq y, inc(y,x) \in Q\\
 0 \text{ if } x\neq y, inc(y,x) \not\in Q
 
  \end{cases} \text{ for } (x,P), (y,Q)\in m^\omega\times \mathfrak{P}$$

\begin{prop}\label{convergenceK1} Fix $(i,j)\in P\in \mathbb{P}$.
\begin{enumerate}
\item If $\{s_0,s_1,\ldots\}\subset m^{<\omega}$ is an $(i,j)$-sequence over $x\in m^\omega$, then $$\lim_k \mathbf{f}_{s_k} = \mathbf{f}_{(x,P)}.$$
\item  If $\{x_0,x_1,\ldots\}\subset m^{\omega}$ is an $(i,j)$-sequence over $x\in m^\omega$, and we choose any $P_k\in \mathbb{P}$, then $$\lim_k \mathbf{f}_{(x_k,P_k)} = \mathbf{f}_{(x,P)}.$$
\end{enumerate}
\end{prop}

\begin{proof}
For statement (1), first notice that $\mathbf{f}_{s_k}(t) \To 1$ if $t\leq x$, and $\mathbf{f}_{s_k}(t)\To 0$ otherwise. So $\mathbf{f}_{s_k}(t) \To \mathbf{f}_{(x,P)}(t)$ for $t\in m^{<\omega}$. Second, notice that $\mathbf{f}_{s_k}(x,P) = 1$ because $inc(x,s_k) = (i,j) \in P$, so we also have $\mathbf{f}_{s_k}(x,P) \To \mathbf{f}_{(x,P)}(x,P)$. Third, if $Q\neq P$, we have $\mathbf{f}_{s_k}(x,Q) =0$ because  $inc(x,s_k) = (i,j) \not\in Q$, so again $\mathbf{f}_{s_k}(x,Q) \To \mathbf{f}_{(x,P)}(x,Q)$. Fourth, if we consider $(y,Q)$ with $x\neq y$, then $\mathbf{f}_{s_k}(y,Q)$ depends on the incidence of $y$ and $s_k$. But from a moment on, this incidence coincides with that of $y$ and $x$, and this makes $\mathbf{f}_{s_k}(y,Q) \To \mathbf{f}_{(x,P)}(y,Q)$ work again. Statement (2) is proven in a similar way. First, $\mathbf{f}_{(x_k,P_k)}(t)$ is 1 if $t\leq x_k$ and 0 otherwise. Notice if $t\leq x$, then $t\leq x_k$ for all but finitely many $k$'s, and if $t\not\leq x$, then $t\not\leq x_k$ for all but finitely many $k$'s. Therefore $\mathbf{f}_{(x_k,P_k)}(t) \To \mathbf{f}_{(x,P)}(t)$. Second, $\mathbf{f}_{(x_k,P_k)}(x,P) = 1$ because $inc(x,x_k) = (i,j) \in P$, so $\mathbf{f}_{(x_k,P_k)}(x,P)\To \mathbf{f}_{(x,P)}(x,P)$.  Third, if $Q\neq P$, then  $\mathbf{f}_{(x_k,P_k)}(x,Q) = 0$ because $inc(x,x_k) = (i,j)\not\in Q$, so again $\mathbf{f}_{(x_k,P_k)}(x,Q) \To \mathbf{f}_{(x,P)}(x,Q)$. Finally, for $y\neq x$, $\mathbf{f}_{(x_k,P_k)}(y,Q)$ depends on $inc(y,x_k)$, but $inc(y,x_k) = inc(y,x)$ for all but finitely many $k$'s. 
\end{proof}

\begin{prop}\label{K1isRosenthal}
$K_1(\mathfrak{P})$ is a separable Rosenthal compactum.
\end{prop}

\begin{proof}
The functions $\mathbf{f}_s$ for $s\in m^{<\omega}$ are clearly continuous on the Polish space $X_{\mathfrak{P}}$. Combining this with Lemms~\ref{densecombs} and Proposition~\ref{convergenceK1},  we get that $\{\mathbf{f}_s : s\in m^{<\omega}\}$ is a countable family of continuous functions on a Polish space with the property that every subsequence contains a further subsequence that converges in $\mathbb{R}^{X_\mathfrak{P}}$. This implies that its closure is a Rosenthal compactum~\cite{BFT}.
\end{proof}

\begin{prop}\label{K1description} The function $X_\mathfrak{P} \To K_1(\mathfrak{P})$ given by $\xi\mapsto \mathbf{f}_\xi$ is a bijection. Thus,
$$K_1(\mathfrak{P}) = \{\mathbf{f}_s : s\in m^{<\omega}\} \cup \{\mathbf{f}_{(x,P)} : (x,P)\in m^\omega\times \mathfrak{P}\}.$$
The points $\mathbf{f}_s$ are isolated in $K_1(\mathfrak{P})$ and the points $\mathbf{f}_{(x.P)}$ are $G_\delta$-points, so $K_1(\mathfrak{P})$ is a first-countable space.
\end{prop}

\begin{proof}
It is easy to see that the function $\xi\mapsto \mathbf{f}_\xi$ is one-to-one. In fact, notice that $\mathbf{f}_{(x,P)} \neq \mathbf{f}_{(x,Q)}$ if $P\neq Q$ because $\mathbf{f}_{(x,P)}(x,Q) = 0$ but $\mathbf{f}_{(x,Q)}(x,Q) = 1$. In any other cases, if $\xi\neq\xi'$ then we can find $t\in m^{<\omega}$ where $\mathbf{f}_\xi(t) \neq \mathbf{f}_{\xi'}(t)$. Since $K_1(\mathfrak{P})$ is a Rosenthal compactum, by the Bourgain-Fremlin-Talagrand theorem, every point $z$ in it is the limit of a sequence in $\{\mathbf{f}_s : s\in m^{<\omega}\}$. By Lemma~\ref{densecombs}, if this sequence is infinite, then it  contains a further subsequence of the form $\{\mathbf{f}_{s_0},\mathbf{f}_{s_1},\ldots\}$, where $\{s_0,s_1,\ldots\}$ is an $(i,j)$-sequence over some $x\in m^\omega$, for some $(i,j)\in P \in \mathbb{P}$. By Proposition~\ref{convergenceK1}, the limit of such sequence is $z = \mathbf{f}_{(x,P)}.$ This proves that the range of the function $\xi\mapsto \mathbf{f}_\xi$  is the full $K_1(\mathbb{P})$. 

The function $\mathbf{f}_s$ is isolated because it is the unique function in $K_1(\mathfrak{P})$ that satisfies $\mathbf{f}_s(s) = 1$, and $\mathbf{f}_s(s^\frown i) = 0$ for all $i\in m$. On the other hand, $\mathbf{f}_{(x,P)}$ is a $G_\delta$-point, because it is the only function in $f\in K_1(\mathfrak{P})$ that satisfies $f(x|_k) = 1$ for all $k<\omega$, and moreover $f(x,P) = 1$.\end{proof}

Finally, we observe that if $T\subset m^{<\omega}$ is a first-move subtree, then the closure of $\{f_t : t\in T\}$ is naturally homeomorphic to the whole $K_1(\mathfrak{P})$.

\section{The spaces $K_\infty(\mathfrak{Q})$}\label{sectionKinfty}

Now, instead of a partition $\mathfrak{P}$ of $m\times m$ we will consider a family $\mathfrak{Q}$ of disjoint subsets of $m=\{0,1,\ldots,m-1\}$. We will construct a separable Rosenthal compact space $K_\infty(\mathfrak{Q})$ with a non-$G_\delta$-point. For this, we consider the Polish space $X_{\mathfrak{Q}} := m^{<\omega}\cup m^{\omega} \times \mathfrak{Q}$, endowed with the same topology that we gave to $X_\mathfrak{P}$ in the previous section.

Now, for every pair $s\in m^{<\omega}$ we consider a function $\mathbf{g}_s: X_{\mathfrak{Q}}\To \{0,1\}$ given by
$$\mathbf{g}_s(t) =  \begin{cases} 1 \text{ if } t= s \\ 0 \text{ otherwise } \end{cases} \text{ for } s,t\in m^{<\omega},$$
$$\mathbf{g}_{s}(y,Q) = \begin{cases}
 1 \text{ if } inc(y,s)=(i,i) \text{ for some } i\in Q\\
 0 \text{ otherwise }
   \end{cases} \text{ for } s\in m^{<\omega}, \ (y,Q)\in m^\omega\times \mathfrak{Q}$$
   
\begin{defn}
The compact space $K_\infty(\mathfrak{Q})$ is the pointwise closure of the functions $\{\mathbf{g}_s : s\in m^{<\omega}\}$ in $\mathbb{R}^{X_{\mathfrak{Q}}}$.
\end{defn}

In order to describe this closure, we consider the function $\mathbf{g}_\infty: X_{\mathfrak{Q}}\To \{0,1\}$ that is constant equal to 0, and, for every $(x,P)\in m^\omega \times \mathfrak{Q}$, the function $\mathbf{g}_{(x,P)}:X_\mathfrak{Q} \To \{0,1\}$ that takes the value 0 at all points except at $(x,P)$, where it takes value 1.

\begin{prop}\label{convergenceKinfty}
Fix $i,j\in m$, and $\{s_0,s_1,\ldots\}\subset m^{<\omega}$ an $(i,j)$-sequence over $x\in m^\omega$
\begin{enumerate}
\item If $i=j\in P\in \mathfrak{Q}$, then $\lim_k \mathbf{g}_{s_k} = \mathbf{g}_{(x,P)}.$
\item If either $i\neq j$ or $i=j\not\in \bigcup \mathfrak{Q}$, then $\lim_k \mathbf{g}_{s_k} = \mathbf{g}_{\infty}.$
\end{enumerate}
On the other hand, the only accumulation point of the set $\{\mathbf{g}_{(x,P)} : x\in m^\omega, P\in\mathfrak{Q}\}$ is $\mathbf{g}_\infty$.
\end{prop}

\begin{proof}
The last statement is evident. It is also clear that any of the two cases considered above, $\lim \mathbf{g}_{s_k}(t) = 0$ for $t\in m^{<\omega}$. On the other hand, $\mathbf{g}_{s_k}(y,Q) = 1$  only if $s_k^\frown \xi \leq y$ with $\xi\in Q$. Hence, $\mathbf{g}_{s_k}(y,Q)$ converges to 0 unless $y=x$ and $i=j\in Q=P$, in which case it converges to 1.
\end{proof}

\begin{prop}\label{KinftyisRosenthal}
$K_\infty(\mathfrak{Q})$ is a separable Rosenthal compactum.
\end{prop}

\begin{proof}
Similar to Proposition~\ref{K1isRosenthal}
\end{proof}

\begin{prop}\label{Kinftydescription}
The function $X_{\mathfrak{Q}} \cup\{\infty\} \To K_\infty(\mathfrak{Q})$ given by $\xi\mapsto \mathbf{g}_\xi$ is a bijection. Thus,$$K_\infty(\mathfrak{P}) = \{\mathbf{g}_s : s\in m^{<\omega}\} \cup \{\mathbf{g}_{(x,P)} : (x,P)\in m^\omega\times \mathfrak{Q}\}\cup \{\mathbf{g}_\infty\}.$$
This is a scattered space of height 3, whose Cantor-Bendixson derivates are
\begin{eqnarray*}
K_\infty(\mathfrak{P})' &=&  \{\mathbf{g}_{(x,P)} : (x,P)\in m^\omega\times \mathfrak{Q}\}\cup \{\mathbf{g}_\infty\}\\
K_\infty(\mathfrak{P})'' &=&  \{\mathbf{g}_\infty\}
\end{eqnarray*}
Thus, the points $\mathbf{g}_s$ are isolated in $K_\infty(\mathfrak{Q})$, the points $\mathbf{g}_{(x.P)}$ are $G_\delta$-points, but if $\mathfrak{Q} \neq\emptyset$, then $\mathbf{g}_\infty$ is not a $G_\delta$-point of $K_\infty(\mathfrak{Q})$.
\end{prop}

\begin{proof}
It is clear that the assignment $\xi\mapsto \mathbf{g}_\xi$ is one-to-one. The fact that this is surjective is deduced from Proposition~\ref{convergenceKinfty} using the same argument as in the proof of Proposition~\ref{K1description}. Each function $\mathbf{g}_s$ is isolated in $K_\infty(\mathfrak{Q})$ as it is the only function such that $\mathbf{g}_s(s) = 1$. The rest of points are not isolated by Proposition~\ref{convergenceKinfty}. On the other hand, it is clear that each $\mathbf{g}_{(x,P)}$ is isolated in  $K_\infty(\mathfrak{Q})'$.
\end{proof}

\section{Minimal spaces}\label{sectionmin}

If we have a finite set $N$ and a surjective function $f:m\times m\To N$, we can associate to it a partition $\mathfrak{P}_f = \{f^{-1}(i) : i\in N\}$ of $m\times m$ into $|N|$ pieces. The following definition is taken from~\cite{stronggaps},

\begin{defn}
For natural numbers $m_0$ and $m_1$, we say that $\varepsilon:m_0^2\To m_1^2$ is a reduction map if there exists $k<\omega$ a one-to-one function $e:m_0 \To m_1^k$ and an element $x\in m_1^{<\omega}$ such that $|x|<k$ and for all $u,v\in m_0$
\begin{enumerate}
\item $\varepsilon(u,v) = inc(e(u),e(v))$ if $u\neq v$,
\item $\varepsilon(u,u) = inc(e(u),x)$ .\\
\end{enumerate}
\end{defn}

\begin{defn}
For two surjective functions $f:m_0^2\To N_0$ and $g:m_1^2\To N_1$, we write $f\prec g$ if there exists a reduction map $\varepsilon:m_0^2\To m_1^2$ such that $f = g\circ \varepsilon$.
\end{defn}

\begin{lem}\label{fprecg}
If $f\prec g$, then $K_1(\mathfrak{P}_f)$ is homeomorphic to a closed subspace of $K_1(\mathfrak{P}_g)$.
\end{lem}

\begin{proof}
We have a reduction map $\varepsilon:m_0^2\To m_1^2$ such that $f = g\circ \varepsilon$. Let $e$ and $x$ be as in the definition of reduction map. As in the proof of \cite[Lemma~11]{stronggaps}, we consider $\phi:m_0^{<\omega}\To m_1^{<\omega}$ defined as
$$\phi(u_0,\ldots,u_k) = e(u_0)^\frown\cdots^\frown e(u_k)^\frown x.$$
and $\psi:m_0^\omega \To m_1^\omega$ defined as
$$\psi(u_0,\ldots,u_k,\ldots) = e(u_0)^\frown\cdots^\frown e(u_k)^\frown \cdots.$$

It is easy to see that $\phi$ and $\psi$ are both injective and that if $\{s_n\}$ is an $(i,j)$-sequence over $x$, then $\{\phi(s_n)\}$ is an $\varepsilon(i,j)$-sequence over $\psi(x)$. The map $\phi$ induces a function $\hat{\phi}:\{\mathbf{f}_t : t\in m_0^{<\omega}\}\To \{\mathbf{f}_t : t\in m_1^{<\omega}\}$ between the dense sets of $K_1(\mathfrak{P}_f)$ and $K_1(\mathfrak{P}_g)$, given by $\hat{\phi}(\mathbf{f}_t) = \mathbf{f}_{\phi(t)}$. We claim that this function extends to a continuous function $\tilde{\phi}:K_1(\mathfrak{P}_f) \To K_1(\mathfrak{P}_g)$. If not, by Lemma~\ref{continuousondense}, we would have two sequences $\{\mathbf{f}_{s_n}\}$ and $\{\mathbf{f}_{t_n}\}$ that converge to the same point in $K_1(\mathfrak{P}_f)$ but $\{\mathbf{f}_{\phi(s_n)}\}$ and $\{\mathbf{f}_{\phi(t_n)}\}$ converge to different points in $K_1(\mathfrak{P}_g)$. By Lemma~\ref{densecombs}, we can suppose that $\{s_n\}$ is an $(i,j)$-sequence over some $y$, and $\{t_n\}$ is a $(u,v)$-sequence over some $z$. Since $\{\mathbf{f}_{s_n}\}$ and $\{\mathbf{f}_{t_n}\}$ converge to the same point, by Proposition~\ref{convergenceK1}, $y=z$ and $(i,j)$ and $(u,v)$ belong to the same piece of the partition $\mathfrak{P}_f$, that is $f(i,j) = f(u,v)$. But then, $\{\mathbf{f}_{\phi(s_n)}\}$ is an $\varepsilon(i,j)$-sequence over $\psi(y)$  and $\{\mathbf{f}_{\phi(t_n)}\}$ is an $\varepsilon(u,v)$-sequence over $\psi(y)=\psi(z)$. Since $$g(\varepsilon(i,j)) = f(i,j) = f(u,v) = g(\varepsilon(u,v)),$$ $\varepsilon(i,j)$ and $\varepsilon(u,v)$ belong to the same piece of the partition $\mathfrak{P}_g$. Therefore, $\{\mathbf{f}_{\phi(s_n)}\}$ and $\{\mathbf{f}_{\phi(t_n)}\}$ converge to the same point $K_1(\mathfrak{P}_g)$. This finishes the proof that we hava a continuous extension $\tilde{\phi}:K_1(\mathfrak{P}_f) \To K_1(\mathfrak{P}_g)$. Moreover, the above argument also proves that $\tilde{\phi}(\mathbf{f}_{(y,f^{-1}(i))}) = \mathbf{f}_{(\psi(y),g^{-1}(i))}$, so $\tilde{\phi}$ is clearly injective.
\end{proof}

The converse of the previous lemma is true, modulo permutations:

\begin{lem}\label{homeomorphiccopy}
Let $g^0:m_0^2\To N_0$ and $g^1: m_1^2\To N_1$ be two surjective functions. If $K_1(\mathfrak{P}_{g^0})$ is homeomorphic to a subspace of $K_1(\mathfrak{P}_{g^1})$, then there exists an injective $\sigma:N_0\To N_1$ such that\footnote{$\sigma \circ g^0$ is not surjective, but we view it as a surjection onto its range.} $\sigma \circ g^0 \prec g^1$. 
\end{lem}

\begin{proof}
Let $\phi:K_1(\mathfrak{P}_{g^0}) \To K_1(\mathfrak{P}_{g^1})$ be a homeomorphic embedding. Let us call $\{\mathbf{f}^0_t : t\in m_0^{<\omega}\}$ and $\{\mathbf{f}^1_t : t\in m_1^{<\omega}\}$ the respective dense sets, and $\{\mathbf{f}^1_{(x,P)} : x\in m_1^{\omega},P\in\mathfrak{P}_{g^1}\}$ the other points of $K_1(\mathfrak{P}_{g^1})$.

Lemma~\ref{strongRamsey} applied to a singleton $A_0$ means that whenever we divide $m^{<\omega}$ into two pieces there exists a first-move subtree contained in one of the pieces. Thus, by passing to a first-move subtree, we can suppose that  
\begin{eqnarray*}\text{(case 1)} & & \text{ either } \{\phi(\mathbf{f}^0_t) : t\in m_0^{<\omega}\}\subset \{\mathbf{f}^1_t : t\in m_1^{<\omega}\}\\
\text{(case 2)} & & \text { or } \{\phi(\mathbf{f}^0_t) : t\in m_0^{<\omega}\}\subset \{\mathbf{f}^1_{(x,P)} : x\in m_1^{\omega},P\in\mathfrak{P}_{g^1}\}.
\end{eqnarray*}
Let us consider first case 1. If we fix $(i,j)$, then the $(i,j)$-combs $A$ of $m_0^{<\omega}$ can be partitioned into $m_1^2+1$ many Borel pieces, depending on whether $\{\phi(\mathbf{f}_t^0 : t\in A\}$ is a $(u,v)$-comb for some $u,v\in m_1$, or it is not a comb at all. By Theorem~\ref{strongRamsey}, after passing to a first-move subtree, we can suppose that all $(i,j)$-combs lie in the same piece of this partition. The piece cannot be that of non-combs because of Lemma~\ref{densecombs}. After repeated application of this argument, we can suppose that for every $(i,j)\in m_0^2$ there exists $\beta(i,j)$ such that if $\{\mathbf{f}^0_t : t\in A\}$ is an $(i,j)$-comb, then $\{\phi(\mathbf{f}^0_t) : t\in A\}$ is a $\beta(i,j)$-comb. The fact that $\phi$ is a homeomorphic embedding, combined with Proposition~\ref{convergenceK1}, implies that $g^0(i0,j0) = g^0(i1,j1)$ if and only if $g^1\beta(i0,j0)=g^1\beta(i1,j1)$. Therefore, we have a well defined injective function $\sigma:N_0\To N_1$ such that $\sigma(g^0(i,j)) = g^1(\beta(i,j))$ for every $(i,j)$. In the language of \cite{stronggaps} that would translate into the fact that $\Gamma_{\sigma\circ g^0}\leq \Gamma_{g^1}$, so by \cite[Lema 11]{stronggaps}, we obtain that $\sigma \circ g^0 \prec g^1$.
The second case,
$$\{\phi(\mathbf{f}^0_t) : t\in 2^{<\omega}\}\subset \{\mathbf{f}^1_{(x,P)} : x\in m_1^{\omega},P\in\mathfrak{P}_{g^1}\},$$
will be reduced to the first case, because we will find a new $\phi'$ that will satisfy the first option. Let $\phi(\mathbf{f}_t^0) = \mathbf{f}^1_{(x_t,P_t)}$. We distinguish two further cases now:

\begin{enumerate}

\item For every $t\in m_0^{<\omega}$ and for every $i\in m_0$ there exists $r\geq t^\frown i$ such that $\{x_t \wedge x_s : s\geq r\}$ is finite. Under this assumption, we can construct inductively a first-move subtree $T$ such that $\{x_t\wedge x_s : s\in T\}$ is finite for every $t\in T$. Then, we can define $\phi'(\mathbf{f}^0_t) = \mathbf{f}^1_{x_t|{k_t}}$ where $k_t<\omega$ is larger than $|x_t\wedge x_s|$ for all $s\in T\setminus\{t\}$. Using Proposition~\ref{convergenceK1} and Lemma~\ref{continuousondense}, we see that the map $\phi(t)\mapsto \phi'(t)$ for $t\in T$ induces a homeomorphism between closures, so $t\mapsto \phi'(t)$ is a new homeomorphic embedding that falls into the previous case, that is: $$\{\phi'(\mathbf{f}^0_t) : t\in T\}\subset \{\mathbf{f}^1_t : t\in m_1^{<\omega}\}.$$

\item There exist $t\in m_0^{<\omega}$ and $i\in m_0$ such that for all $r\geq t^\frown i$, the set $\{x_t\wedge x_s : s\geq r\}$ is infinite. In this case, it is possible to constuct inductively a first-move subtree $T\subset m_0^{<\omega}$ rooted at $t^\frown i$ such that the map $s\mapsto x_t\wedge x_s$ is injective on $T$. Then, using Proposition~\ref{convergenceK1} in combination with Lemma~\ref{densecombs}, we see that the only possible limits of subsequences of $\{\phi'(\mathbf{f}^0_t) : t\in T\}$ lie in the finite set $\{\mathbf{f}^1_{(x_t,P)} : P\in \mathfrak{P}_{g^1}\}$. This is impossible because $\phi$ is a homeomorphic embedding and $\{\mathbf{f}^0_t : t\in T\}$ has uncountably many cluster points.
\end{enumerate}

\end{proof}

\begin{lem}\label{nsubset}
For every surjective $g:m_1^2\To N_1$ and for every $n_0<|N_1|$, there exist $N_0\subset N_1$ of cardinality $n_0$ and a surjective $f:m_0^2\To N_0$ such that $f\prec g$.
\end{lem}

\begin{proof}
Consider $$M = \{p\in N_1 : \exists i\neq j : g(i,j) = p\}.$$ We distinguish two cases. The first case is that $|M|\leq n_0$.  Without loss of generality, we can suppose that $M$ is of the form $M=\{0,1,\ldots,\xi\}$. For $i\in M$, find $u_i,v_i\in m_1$, $u_i\neq v_i$ such $g(u_i,v_i) = i$. For $i\in n_0\setminus M$, find $w_i\in m_1$ such that $g(w_i,w_i)=i$. We consider $m_0=n_0= N_0$ and the reduction map $\varepsilon:m_0^2\To m_1^2$ induced by:
\begin{itemize}
\item $x = (v_0,\ldots, v_\xi)$,
\item $e(i) = (v_0,\ldots,v_{i-1},u_i,0,0,\ldots,0)$ for $i\in M$,
\item $e(i) =(v_0,\ldots, v_\xi,w_i)$, for $i\in m_0\setminus M$.
\end{itemize}
In this way, we get a function $f=g\circ \varepsilon\prec g$. Notice that all posible incidences of the $e(i)$'s and $x$ of the form $(u,u)$ are necessarily of the form $(w_i,w_i)$ for $i\in n_0\setminus M$. This implies that the range of $f$ is contained in $N_0$. On the other hand, if $i\in M$, then $f(i,i) = g(inc(e(i),x) = g(u_i,v_i) = i$, and if $i\in m_0\setminus M$, then $f(i,i) = g(inc(e(i),x) = (w_i,w_i) = i$, so $f:m_0^2\To N_0$ is onto. 

The second case is that $|M|>n_0$. In this case we only need to deal with pairs $(i,j)$ with $i\neq j$. We will need to find suitable $m_0$, $k$ and $x$ and $e:m_1\To m_0^{<k}$ so that the induced reduction map $\varepsilon$ has the propert that the range $g\circ \varepsilon$ has cardinality $n_0$. The point $x$ will be of the form $x=(j_0,j_1,\ldots,j_m)$, and the function $e$ of the form $e(r) = (j_0,\ldots,j_{r-1},i_r,0,\ldots,0)$. In this way, the range of $g\circ \varepsilon$ is exactly the set 
$$\{g(i_0,j_0),g(j_0,i_0), g(i_1,j_1), g(j_1,i_1),\ldots, g(i_{m-1},j_{m-1}), g(j_{m-1},i_{m-1}), g(i_m,j_m)\}.$$
Notice the exclusion of $g(j_m,i_m)$ from this set. It is easy to obtain integers making the above set to have cardinality exactly $n_0$. We can keep adding pairs $(i,j)$ till we get a set
$$\{g(i_0,j_0),g(j_0,i_0),\ldots, g(i_{m-1},j_{m-1}), g(j_{m-1},i_{m-1})\}$$ 
of cardinality either $n_0$ or $n_0-1$. If it is $n_0$, then declare $i_m=i_0$, $j_m=j_0$. If otherwise that cardinality is $n_0-1$, then pick $(i_m,j_m)$ such that $g(i_m,j_m)$ is a new value.
\end{proof}

Given a set $A$, we use the notation $\langle A\rangle ^2 = \{(i,j) \in A^2 : i\neq j\}$.

We fix a natural number $n$.  A strong-dense-type is a collection of data of the form $\alpha = (A,B,C,D,E,\psi,\mathcal{P},\gamma)$ where
\begin{enumerate}
\item $n=A\cup B\cup C\cup D\cup E$ is a partition of $n$.
\item $\psi: \langle A \rangle ^2 \To B$ is a surjective function.
\item $\mathcal{P}$ is a partition of $C$ into sets of cardinality either 1 or 2.
\item $\gamma: D\To B\cup E$ is a function such that $|\gamma^{-1}(k)|\geq 2$ for every $k\in E$.
\end{enumerate}
with the additional property that if $A=\emptyset$ then $B=D=E=\emptyset$ and all elements of $\mathcal{P}$ have cardinality 2. This definition is the same as \cite[Definition 19]{stronggaps} except that we do not allow the value $\infty$ for $\psi$ and $\gamma$ (this is because here we are only interested in dense gaps). We write $|\alpha| = n$.

To each such set of data $\alpha$, we can associate a surjective function $f^\alpha:m\times m \To n$, and the correspondig partition $\mathfrak{P}_{f^\alpha}$ of $m\times m$, for certain $m$. It is convenient to consider a finite set $M$ rather than a natural number $m$. The set $M$ can be later identified with $m = |M|$ through enumeration. This set is $M = A^\ast\cup \mathcal{P}\cup D$ where $A^\ast = A$ if $A\neq \emptyset$, and $A^\ast = \{0\}$ if $A=\emptyset$. In order to define $f^\alpha:M^2\To n$ we need some further notation. For $a\in \mathcal{P}$, let $\sigma(a) = \min(a)$ and $\tau(a) = \max(a)$ so that $a = \{\sigma(a),\tau(a)\}$ for every $a\in\mathcal{P}$. We define $\sigma(k) = k$ for $k\in A^\ast\cup D$, and $\tau(k) = \gamma(k)$ for $k\in D$. Notice that $\sigma:M\To n$ and $\tau:\mathcal{P}\cup D \To n$. The function $f = f^\alpha$ is defined as:
\begin{enumerate}
\item $f(i,i) = \sigma(i)$ for $i\in M$;
\item $f(i,j) = \psi(i,j)$ for $i,j\in A$, $i\neq j$;
\item $f(i,j) = \sigma(i)$ if $i\in M\setminus A^\ast$ and ($j\in A^\ast$ or $\sigma(i)<\sigma(j)$).
\item $f(i,j) = \tau(j)$ if $j\in M\setminus A^\ast$ and ($i\in A^\ast$ or $\sigma(i)>\sigma(j)$).\\
\end{enumerate}

\begin{lem}\label{minimal}
If $n\leq |\mathfrak{P}|$, then $K_1(\mathfrak{P})$ contains a homeomorphic copy of a space of the form $K_1(\mathfrak{P}_{f^\alpha})$ for some strong-dense-type $\alpha$ with $|\alpha| = n$.
\end{lem}

\begin{proof}
By Lemma~\ref{fprecg} and Lemma~\ref{nsubset}, $K_1(\mathfrak{P})$ contains $K_1(\mathfrak{P}_f)$ for some surjective $f:m^2\To n$. It follows from the results of~\cite{stronggaps}, that there exists $$\alpha = (A,B,C,D,E,\psi,\mathcal{P},\gamma)$$ with $|\alpha|=n$ such that $f^\alpha\prec f$.
\end{proof}

\section{The open degree}\label{sectionopen}

In this section, we study the open degree $odeg(K)$ of Definition~\ref{defodeg}, and we compute $odeg(K_1(\mathfrak{P}))$ and $odeg(K_\infty(\mathfrak{Q}))$. A first observation is that, by a standard compactness argument, one can substitute points by compact sets.

\begin{lem}\label{separationofclosed}
If $odeg(K)\leq n$, then there exists a countable family $\mathcal{F}$ of open sets such that for every pairwise disjoint compact subsets $L_0,\ldots,L_n\subset K$ there exists $V_0,\ldots,V_n\in \mathcal{F}$ such that $L_i\subset V_i$ and $\bigcap_i V_i = \emptyset$.
\end{lem}

\begin{proof}
Start with a countable family $\mathcal{F}$ of open sets that satisfies the definition of $odeg(K)$, meaning that for every different $x_0,\ldots,x_n\in K$ there exists $V_0,\ldots,V_n\in \mathcal{F}$ such that $x_i\in V_i$ and $\bigcap_i V_i = \emptyset$. We can suppose that $\mathcal{F}$ is closed under finite unions and intersections. With this additional assumption, we see that the family $\mathcal{F}$ now satisfies the condition stated in the lemma. We proceed by induction on $p = |\{i<n : |L_i|>1\}|$. The case $p=0$ is clear. If we are in the case $p$, we are given $L_1,\ldots,L_n$ where $|L_i|=1$ for $i>p$, and we want to find $V_i\in \mathcal{F}$ with $L_i\subset V_i$ and $\bigcap V_i = \emptyset$. By the inductive hypothesis, for every $x\in L_p$ there exist $V_i^x\in \mathcal{F}$ such that $x\in V_p^x$, $L_i\subset V_i^x$ and $\bigcap_i V_i^x = \emptyset$. By compactness there exists a finite set $G\subset L_p$ such that $L_p\subset\bigcup_{x\in G} V^x_p$. It is finally enough to consider $V_p = \bigcup_{x\in G}V^x_p$ and $V_i = \bigcap_{x\in G} V_i^x$ for all $i\neq p$.
\end{proof}

Notice that $odeg(K) = 0$ only in the trivial case when $K=\emptyset$. The next case is the following:

\begin{prop}
If $K\neq \emptyset$, then $odeg(K) = 1$ if and only if $K$ is metrizable.
\end{prop}

\begin{proof}
Remember that $K$ is metrizable if and only if it has a countable basis of open sets. So if $K$ is metrizable, then the countable basis $\mathcal{F}$ shows that $odeg(K)\leq 1$. Coversely, if $odeg(K)\leq 1$, then by Lemma~\ref{separationofclosed}, $K$ has a countable family $\mathcal{F}$ of open sets such that every two disjoint closed sets are separated by disjoint elements of $\mathcal{F}$. This implies that $\mathcal{F}$ is a countable basis for the topology of $K$, so $K$ is metrizable.
\end{proof}

The following related notion (in the case $n=2$) was relevant in the analysis of \cite{1BC}:

\begin{defn}
A compact space $K$ is a \emph{premetric compactum of degree at most $n$} if  there exists a continuous function $h:K\To M$ onto a metrizable compactum such that $|h^{-1}(x)|\leq n$ for every $x\in M$.
\end{defn}

\begin{prop}\label{premetricdegree}
$K$ is a premetric compactum of degree at most $n$ if and only if there exists a countable family $\mathcal{F}$ of open $F_\sigma$-sets such that for every $x_0,\ldots,x_n\in K$ there exist $V_i\in\mathcal{F}$ with $x_i\in V_i$ and $\bigcap_i V_i = \emptyset$. In particular, if $K$ is a premetric compactum of degree at most $n$, then $odeg(K)\leq n$.
\end{prop}

\begin{proof}
If $h:K\To M$ is as above, it is enough to consider $\mathcal{F} = \{h^{-1}(W) : W\in\mathcal{B}\}$ where $\mathcal{B}$ is a countable basis for the topology of $M$. Conversely, if we have such a family $\mathcal{F}$, for each $W\in \mathcal{F}$ we can find a continuous function $h_W:K\To [0,1]$ such that $h_W^{-1}(0,1] = W$, and then we consider $h:K \To [0,1]^\mathcal{F}$ given by $h(x) = (h_W(x))_{W\in\mathcal{F}}$, and $M = h(K)$.
\end{proof}

As we will show later, the converse of the last statement of  Proposition~\ref{premetricdegree} is not true.

\begin{thm}\label{odegcomputation}
$odeg(K_1(\mathfrak{P})) = |\mathfrak{P}|$ and $odeg(K_\infty(\mathfrak{Q})) = |\mathfrak{Q}|+1$.
\end{thm}

\begin{proof}
We prove the four inequalities.

$odeg(K_1(\mathfrak{P})) \leq |\mathfrak{P}|$. For this, it is enough to consider the clopen sets $$V_t = \{f\in K_1(\mathfrak{P}) : f(t) = 1\} = \{\mathbf{f}_s : t\leq s\}\cup \{\mathbf{f}_{(x,P)} : t\leq x\}, $$
$$W_t = \{\mathbf{f}_t\} = \{f\in K_1(\mathfrak{P}) : f(t)=1, f(t^\frown i)=0, i<m\}$$ for $t\in m^{<\omega}$. Let us check that the family $\mathcal{F} = \{V_t, W_t, K\setminus W_t : t\in m^{<\omega}\}$ has the desired property. So we declare $n=|\mathfrak{P}|$ and we pick points $x_0,\ldots,x_n\in K_1(\mathfrak{P})$. If at least one of the points is of the form $\mathbf{f}_t$ with $t\in m^{<\omega}$, then we can simply take $W_t$ as its neighborhood, and $K\setminus W_t$ as the neighborhood of the rest. Otherwise, if all points are of the form $x_i = \mathbf{f}_{(y_i,P_i)}$, then since $|\mathfrak{P}| = n$, we can find $i,j$ such that $y_i\neq y_j$. Then we can pick  $s,t\in n^{<\omega}$ such that $y_i\wedge y_j < s$, $y_i\wedge y_j<t$, $s\leq y_i$, $t\leq y_j$. In that case $V_s$ is a neighborhood of $\mathbf{f}_{(y_i,P_i)}$, $V_t$ is a neighborhood of $\mathbf{f}_{(V_j,P_j)}$ and $V_s\cap V_t = \emptyset$.

$odeg(K_\infty(\mathfrak{Q})) \leq |\mathfrak{Q}|+1$. For this, we consider the sets
$$V_t = \{\mathbf{g}_s : t\leq s\} \cup \{\mathbf{g}_{(x,P)} : t\leq x\} ,$$
$$W_t = \{\mathbf{g}_t\} = \{g\in K_\infty(\mathfrak{Q}) : g(t) = 1\}.$$
Notice that $W_t$ is a clopen set, and that $V_t$ is open because
$$V_t = \bigcup_{t\leq y, P\in\mathfrak{Q}} \{ g \in K_\infty(\mathfrak{Q}) : g(y,P) = 1\} \setminus \bigcup_{s<t}W_s.$$
We check that $$\mathcal{F} = \{V_t,W_t , K\setminus W_t: t\in n^{<\omega}\}$$
satisfies that whenever we pick $x_0,\ldots,x_n\in K_\infty(\mathfrak{Q})$, with $n=|\mathfrak{Q}|+1$, we can find neighborhoods of each point in the family $\mathcal{F}$ which are disjoint. If at least one of the points is of the form $\mathbf{g}_t$ with $t\in m^{<\omega}$, then we can take $W_t$ and $K\setminus W_t$ as neighborhoods. Otherwise, all points are either of the form $x_i = \mathbf{g}_{(y_i,P)}$ or $\mathbf{g}_\infty$. Similarly as before, since $n=|\mathfrak{Q}|+1$, we can find  $i,j$ such that $y_i\neq y_j$. Then we can pick incomparable $s,t\in n^{<\omega}$ such that $s,t > y_i\wedge y_j$, $s\leq y_i$, $t\leq y_j$, and in that case $V_s$ is a neighborhood of $\mathbf{g}_{(y_i,P_i)}$, $V_t$ is a neighborhood of $\mathbf{g}_{(y_j,P_j)}$ and $V_s\cap V_t = \emptyset$.

$odeg(K_1(\mathfrak{P})) \geq |\mathfrak{P}|$. We suppose that we have a countable family $\mathcal{F}$ of open subsets of $K_1(\mathfrak{P})$ that satisfies the property of the definition of open degree for $n=|\mathfrak{P}|-1$ and we work towards a contradiction. Consider the family $$\mathcal{F}' = \{ \{s\in m^{<\omega} : \mathbf{f}_s\in V\} : V\in \mathcal{F}\}.$$
Claim: If we have $\{A_P\subset m^{<\omega} : P\in\mathfrak{P}\}$ such that $A_P$ is an $(i,j)$-sequence for some $(i,j)\in P$, then there exist sets $\{B_P : P\in \mathfrak{P}\}\subset \mathcal{F}'$ such that $A_P \setminus B_P$ is finite and $\bigcap_{P\in\mathfrak{P}} B_P = \emptyset$.

Proof of the claim: By Proposition~\ref{convergenceK1}, $\{\mathbf{f}_s : s\in A_P\}$ converges to a point of the form $\mathbf{f}_{(x_P,P)}$. We can find open sets $V_P\in \mathcal{F}$ such that $\mathbf{f}_{(x_P,P)}\in V_P$ and $\bigcap_P V_P = \emptyset$. It is enough to take $B_P = \{s\in m^{<\omega} : \mathbf{f}_s\in V_P\}\in \mathcal{F}'$. This finishes the proof of the claim.

The claim above implies that, if we denote  $$\mathcal{I}_P = \{A\subset m^{<\omega} : A \text{ is an }(i,j)\text{-comb for some }(i,j)\in P\},$$
we have that the families $\{\mathcal{I}_P : P\in\mathfrak{P}\}$ are countably separated, in contradiction with Lemma~\ref{notcountablyseparated}.

$odeg(K_\infty(\mathfrak{Q})) \geq |\mathfrak{Q}|+1$.  We suppose that we have a countable family $\mathcal{F}$ of open subsets of $K_\infty(\mathfrak{Q})$ that satisfies the property of the definition of open degree for $n=|\mathfrak{Q}|$ and we work towards a contradiction. Consider the family $$\mathcal{F}' = \{ \{s\in m^{<\omega} : \mathbf{g}_s\in V\} : V\in \mathcal{F}\}.$$
Claim: If we have $\{A_P\subset m^{<\omega} : P\in\mathfrak{Q}\cup\{\infty\}\}$ such that $A_P$ is an $(i,i)$-sequence for some $i\in P$, and $A_\infty$ is a $(0,1)$-comb, then there exist sets $\{B_P : P\in \mathfrak{Q}\cup\{\infty\}\}\subset \mathcal{F}'$ such that $A_P \setminus B_P$ is finite and $\bigcap_P B_P = \emptyset$.

Proof of the claim: By Proposition~\ref{convergenceKinfty}, $\{\mathbf{g}_s : s\in A_P\}$ converges to a point of the form $\mathbf{g}_{(x_P,P)}$ for $P\in \mathfrak{Q}$, and  $\{\mathbf{g}_s : s\in A_\infty\}$ converges to $\mathbf{g}_{\infty}$. Then, we can find open sets $V_P\in \mathcal{F}$ such that $\mathbf{g}_{(x_P,P)}\in V_P$, $\mathbf{g}_\infty\in V_\infty$ and $\bigcap\{ V_P : P\in \mathfrak{Q}\cup \{\infty\}\} = \emptyset$. It is enough to take $B_P = \{s\in m^{<\omega} : \mathbf{g}_s\in V_P\}\in \mathcal{F}'$. This finishes the proof of the claim.

The claim above implies that, if we denote  $$\mathcal{I}_P = \{A\subset m^{<\omega} : A \text{ is an }(i,i)\text{-comb for some }i\in P\},$$
for $P\in \mathfrak{Q}$, and $\mathcal{I}_\infty$ is the family of all $(0,1)$-combs of $m^{<\omega}$, then the families $\{\mathcal{I}_P : P\in\mathfrak{Q}\cup \{\infty\}\}$ are countably separated, in contradiction with Lemma~\ref{notcountablyseparated}.
\end{proof}

\begin{prop}
$K_1(\mathfrak{P})$ is a premetric compactum of degree $|\mathfrak{P}|$ but if $\mathfrak{Q}\neq \emptyset$, then $K_\infty(\mathfrak{Q})$ is not premetric of any degree.
\end{prop}

\begin{proof}
The first statement follows from Proposition~\ref{premetricdegree} because the sets $V_t$, $W_t$ and $K\setminus W_t$ in the first part of the proof of Theorem~\ref{odegcomputation} are all clopen. The second statement follows from the fact that $K_\infty(\mathfrak{Q})$ has a point $\mathbf{g}_\infty$ that is not $G_\delta$-point.
\end{proof}

A compact space $K$ is a premetric compactum of degree $n$ if it is a premetric compactum of degree at most $n$ but not a premetric compactum of at most $n-1$. The main result of \cite{APT} is that, for every $n$, there exist two Rosenthal premetric compacta $S_n(I)$ and $D_n(2^\mathbb{N})$ of degree $n$ such that every separable Rosenthal premetric compactum of degree $n$ contains a homeomorphic copy of either $S_n(I)$ and $D_n(2^\mathbb{N})$. Although there is some superficial similarity, the result from \cite{APT} is not deduced from the results of this paper, nor vice-versa. It is worth to notice some differences:
\begin{itemize}
\item The spaces $S_n(I)$ and $D_n(2^\mathbb{N})$ from \cite{APT} are not separable, so that is not a \emph{basis} result,
\item While $S_n(I)$ and $D_n(2^\mathbb{N})$ are just two spaces for every $n$, the number of spaces in the list of Theorem~\ref{finitebasis} increases with $n$,
\item All premetric compacta of finite degree are first countable. However, non-first countable spaces, like the $K_\infty(\mathfrak{Q})$, may have finite open degree.
\item Even for first-countable separable Rosenthal compacta, the two degrees might be quite different. For example, let $\mathfrak{P}=\{P_1,\ldots,P_{m-1}\}$ be a partition of $m\times m$ where $P_k = \{(i,j) : \max\{1,i,j\} = k\}$. Inside $K_1(\mathfrak{P})$, consider
\begin{eqnarray*}
L &=& \overline{\{\mathbf{f}_{s^\frown i} : s\in 2^{<\omega}, i\in m\}}\\
 &=& \{\mathbf{f}_{s^\frown i} : s\in 2^{<\omega}, i\in m\} \cup \left\{\mathbf{f}_{(x,P)} : x\in 2^\omega, P\in\mathfrak{P}\right\}
\end{eqnarray*}
We have that $odeg(L)=2$. As countable family of open sets to witness this, we can take $$\mathcal{F} = \{W_t,L\setminus W_t, L\cap V_t, U_i : t\in 2^{<\omega},i\in \{2,\ldots,m\}\}$$ where $W_t$ and $V_t$ are as in the first part of the proof of Theorem~\ref{odegcomputation}, and $$U_i = \{\mathbf{f}_{s^\frown i} : s\in 2^{<\omega}\} \cup \left\{\mathbf{f}_{(x,P_i)} : x\in 2^\omega\right\}.$$
On the other hand, $L$ is a premetric compactum of degree $m-1$. In fact $\left\{\mathbf{f}_{(x,P)} : x\in 2^\omega, P\in\mathfrak{P}\right\}$ is homeomorphic to the space $D_{m-1}(2^\mathbb{N})$ of \cite{APT}.

\end{itemize}

\section{The main result}\label{sectionmain}

\begin{lem}\label{findseparatedK}
Let $K$ be a separable Rosenthal compact space with $odeg(K)\geq n$ and $D$ a countable dense subset of $K$. Then there exists an injective map $d:n^{<\omega}\To D$ such that
$$(\star\star)\ \ \ \overline{\{d[t] : t^\frown i \leq z\}} \cap \overline{\{d[t] : t^\frown j \leq z\}} = \emptyset$$
for all $z\in n^{\omega}$ and all $i<j<n$.
\end{lem}

\begin{proof}
By Lemma~\ref{densesetofcontinuous}, we can suppose that $K$ is a compact set of first Baire class on the Polish space $\omega^\omega$ with a dense countable set of continuous functios $D\subset K$. We will further assume that $f(r_0,r_1,r_2,\ldots) = -f(r_0+1,r_1,r_2,\ldots)$ for all $f\in K$ and all $(r_0,r_1,\ldots)\in \omega^\omega$. This can be easily done by adding a new coordinate to $\omega^\omega$ at the beginning. The reason for this is to avoid consideration of too many cases in the future because now we will have that if $f,g\in K$ and $f\neq g$, then there exists $x\in\omega^\omega$ such that $f(x)<g(x)$.

We consider an infinite game of two players in the sense of \cite[Section 20]{Kechris}. At stage $k<\omega$ Player I plays $(d_k,s^{ij}_k: i<j<n)$ where $d_k\in D$, $s^{ij}_k\in \omega^{<\omega}$ and, if $k>0$, $s^{ij}_{k-1}< s^{ij}_k$. Player II responds with an integer $p_k\in \{0,\ldots,n-1\}$. At the end of the game, consider $x^{ij}\in \omega^\omega$ to be the branch determined by the $s^{ij}_k$'s. Player I wins if and only if for every $i<j<n$ there exist two rational numbers $q<q'$ and a natural number $k_0$ such that 
$$ (\star) \ \ \{d_k(x^{ij}) : k>k_0, p_k=i\} < q <q' < \{d_k(x^{ij}) : k>k_0, p_k=j\}$$ 

Since the $d_k's$ are continuous, the last statement $(\star)$ can be rephrased by saying that for every $k>k_0$ there there exists a further $k_1>k_0$ such that $d_k(x)<q$ whenever $s^{ij}_{k_1}\leq x$, $p_k=i$, and such tht $d_k(y)>q'$ whenever $s^{ij}_{k_1}\leq y$, $p_k=j$. Rephrased in this way, it is clear that this is a Borel game. By Martin's Theorem, cf. \cite[Theorem 20.5]{Kechris}, one of the two players has a winning strategy.

If Player II has a winning strategy, then we claim that $odeg(K)<n$. For every finite partial round of the game $\xi = (d_0,s_0^{ij},p_0,\ldots,d_k,s_k^{ij},p_k)$, every $(s^{ij} : i<j<n)$ with $s_k^{ij}<s^{ij}$ and every $p\in\{0,\ldots,n-1\}$, consider $D[\xi,(s^{ij}),p]$ the set of all $d\in D$ such that Player II, according to its strategy, plays $p$ after $\xi$ is played and Player I plays  $(d,s^{ij},i<j<n)$. We claim that the countable family $\mathcal{F}$ of all open sets of the form $$V[\xi, (s^{ij}), p] = K \setminus \overline{D[\xi,(s^{ij}),p]}$$
witnesses that $odeg(K)<n$. So, let us take $f_0,\ldots,f_n\in K$, and we shall find disjoint neighborhoods from $\mathcal{F}$. For every $i<j$ we can find $x^{ij}\in \omega^{\omega}$ and $q_{ij} < q'_{ij}$ such that $f_i(x^{ij}) < q_{ij} < q'_{ij} < f_j(x^{ij})$. 

For every $i$, let $$W(i) = \{f\in K : f(x^{ij})<q_{ij} \text{ for } j>i\text{ and } f(x^{ji}) > q'_{ji} \text{ for }j<i\}$$

If Player I is able to play all the time in such a way that $s^{ij}_k = x^{ij}|_k$, and, when Player II is playing according to its strategy, $d_k \in W(p_k)$, then Player I will win, a contradiction. Therefore Player I cannot play like this all the time, and this implies  that there is a finite round of the game $\xi = (d_0,s_0^{ij},p_0,\ldots,d_k,s_k^{ij},p_k)$ where $s_l^{ij} = x^{ij}|_ l$ such that
$$D[\xi,(x^{ij}|_{k+1}),p] \cap W(p) = \emptyset$$
for every $p\in n$. Since $W(p)$ is open, this implies that
$$\overline{D[\xi,(x^{ij}|_{k+1}),p]} \cap W(p) = \emptyset,$$
hence
$$W(p) \subset V[\xi,(x^{ij}|_{k+1}),p].$$
On the one hand, $f_p\in W(p)$  and on the other hand $\bigcup_{p<n} D[\xi,(x^{ij}|_{k+1}),p] = D$, therefore $\bigcap_{p<n}V[\xi,(x^{ij}|_{k+1}),p]  = \emptyset$. This finishes the proof that $odeg(K)<n$. 

So we suppose that Player I has a winning strategy. For every $t = (t_0,\ldots,t_k)\in n^{<\omega}$, we consider $(d[t], s^{ij}[t])$. the $k$-th move of Player I according to its strategy after Player II has played $t_0,t_1,\ldots,t_k$. Since Player I always wins when playing with his strategy, using property $(\star)$ after Player II plays the integers in any $z \in n^\omega$, we get
$$(\star\star)\ \ \ \overline{\{d[t] : t^\frown i \leq z\}} \cap \overline{\{d[t] : t^\frown j \leq z\}} = \emptyset.$$ 

This is not the end because $d:n^{<\omega}\To D$ might not be injective. However, we can define a first-move embedding $\sigma:n^{<\omega}\To n^{<\omega}$ inductively such that $\sigma(t^\frown i) = \sigma(t)^\frown i ^\frown s_t$ where $s_t$ is chosen so that $d(\sigma(t^\frown i))$ is different from all previously defined values of $d\circ \sigma$. Notice that $(\star\star)$ implies that such an $s_t$ must exist. In this way $d\circ \sigma:n^{<\omega}\To n^{<\omega}$ would be a new function which still satisfies $(\star\star)$ and is moreover injective.
\end{proof}

\begin{lem}\label{findstandardK}
For every injective funciton $d:n^{<\omega}\To K$ from $n^{<\omega}$ into a Rosenthal compactum $K$ 
there exists a first-move embedding $\sigma:n^{<\omega}\To n^{<\omega}$  such that
\begin{enumerate}
\item either the bijection $d(\sigma(t))\mapsto \mathbf{f}_t$ extends to a homeomorphism between $\overline{d(\sigma(n^{<\omega}))}$ and a space $K_1(\mathfrak{P})$ for some partition $\mathfrak{P}$ of $n\times n$
\item or the bijection $d(\sigma(t))\mapsto \mathbf{g}_t$ extends to a homeomorphism between $\overline{d(\sigma(n^{<\omega}))}$ and a space $K_\infty(\mathfrak{Q})$ for some family $\mathfrak{Q}$ of pairwise disjoint subsets of $n$.
\end{enumerate}
\end{lem}

\begin{proof}
Again, by Lemma~\ref{densesetofcontinuous}, we can suppose that $D = d(n^{<\omega})$ is a family of continuous functions on $\omega^\omega$ and $K$ is its pointwise closure. The image of a $t\in n^{<\omega}$ under the function $d$ will be denoted by $d_t$, while $d_t(x)$ will be the evaluation of the function $d_t$ on some $x\in \omega^\omega$. First of all, observe that the family
$$\mathfrak{A} = \{a\subset n^{<\omega} : \{d_t : t\in a\}\text{ is a convergent sequence}\}$$
is a coanalytic family of subsets of $n^{<\omega}$.
This is because $a\in \mathfrak{A}$ if and only if for every $x\in \omega^\omega$ and for every rational $\varepsilon>0$ there exists a finite set $F\subset n^{<\omega}$ such that $|d_t(x) - d_s(x)|<\varepsilon$ for all $t,s\in a\setminus F$, and we are supposing that the functions $d_t$ are continuous on $\omega^\omega$.  In particular $\mathfrak{A}$ is Baire-measurable, so we can apply Theorem~\ref{strongRamsey}, and we conclude that for any given infinite $A\subset n^{<\omega}$ there exists a first move subtree $T_A\subset n^{<\omega}$ such that either $\{d_t : t\in a\}$ is convergent whenever $A\approx a\subset T_A$ or $\{d_t : t\in a\}$ is never convergent whenever  $A\approx a\subset T_A$. By passing to a first-move subtree after applying this principle finitely many times, we can suppose that:

\begin{enumerate}
\item For every $i,j,k,l\in n$, either $\{d_t : t\in a\}$ converges for every $(i,j,k,l)$-double comb $a$, or $\{d_t : t\in a\}$ never converges for any $(i,j,k,l)$-doublecomb $a$.
\item  For every $i,j,k,l,u,v\in n$, either $\{d_t : t\in a\}$ converges for every $(u,v)$-split $(i,j,k,l)$-double comb $a$, or $\{d_t : t\in a\}$ never converges for any $(u,v)$-split $(i,j,k,l)$-doublecomb $a$.
\item For every $i,j\in n$, the sequence $\{d_t : t\in a\}$ converges for every $(i,j)$-comb $a$.
\end{enumerate}  

Notice that in the case of $(i,j)$-combs the case of never convergence cannot occur because any infinite subset of an $(i,j)$-comb is an $(i,j)$-comb, and every sequence in $K$ has a convergent subsequence.

We claim that for every $y\in n^\omega$ and for every $(i,j)\in n\times n$, the limit $\lim\{d_t : t\in a\}$ is the same for all $(i,j)$-combs over the branch $y$, and we can call this limit $\mathbf{h}_{(z,i,j)}$. This follows from the fact that if we have two $(i,j)$-combs $a$ and $b$ over the same branch, then it is easy to construct a new $(i,j)$-comb of the form $a'\cup b'$ where $a'\subset a$ and $b'\subset b$ are infinite.

We consider an equivalence relation on $n\times n$ by declaring that $(i,j)\sim (k,l)$ if $\{d_t : t\in a\}$ converges for all $(i,j,k,l)$-double combs. This is an equivalence relation because we can produce an $(i,j,k.l)$-double comb $a$ and an $(k,l,p,q)$-double comb $b$ that contain a common $(k,l)$-comb. Thus, if the first two double combs were convergent, they have to converge to the same limit, and then we can produce an $(i,j,p,q)$-double comb that converges to that same limit by joining an $(i,j)$-comb inside $a$  and a $(p,q)$-comb inside $b$. 

Let $\mathfrak{P}$ be the partition of $n\times n$ associated to the equivalence relation $\sim$. It is clear that $\mathbf{h}_{(z,i,j)} = \mathbf{h}_{(z,k,l)}$ if $(i,j)\sim (k,l)$. So we can rename the point $\mathbf{h}_{(z,i,j)}$ as $\mathbf{h}_{(z,P)}$, where $P\in\mathfrak{P}$ is the equivalence class for which $(i,j)\in P$. We distinguish two cases.

\textbf{Case 1.} For all $i,j,k,l,u,v\in n$ with $u\neq v$, the sequence $\{d_t : t\in a\}$ never converges for any $(u,v)$-split $(i,j,k,l)$-doublecomb $a$. First of all, this implies that $\mathbf{h}_{(x,P)} \neq \mathbf{h}_{(y,Q)}$ whenever $x\neq y$, because in other case, we cound construct a $inc(x,y)$-split $(i,j,k,l)$-double comb $a$ such that $\{d_t : t\in a\}$ converges. We claim that the bijective map $d_t \mapsto \mathbf{f}_t$ extends to a homeomorphisms between their closures. By Lemma~\ref{continuousondense}, if the map $d_t \mapsto \mathbf{f}_t$ does not extend to a continuous function between their closures, then there exist two sequences $a,b\subset n^{<\omega}$ such that $\{d_t : t\in a\}$ and $\{d_t : t\in b\}$ converge to the same point but $\{\mathbf{f}_t : t\in a\}$ and $\{\mathbf{f}_t : t\in b\}$ converge to different points. We can suppose, by Lemma~\ref{densecombs} that $a$ is an $(i,j)$-comb over some $x$ and $b$ is a $(k,l)$-comb over some $y$. Since $\{d_t : t\in a\}$ and $\{d_t : t\in b\}$ converge to the same point, we have that $x=y$ and there exists $P$ such that $(i,j),(k,l)\in P$. This implies that $\{\mathbf{f}_t : t\in a\}$ and $\{\mathbf{f}_t : t\in b\}$ also converge to the same point, by Proposition~\ref{convergenceK1}. Exactly the same argument shows that the inverse map $\mathbf{f}_t\To d_t$ extends to a continuous function between their closures, hence this continuous function is a homeomorphism.

\textbf{Case 2.} There exist $i,j,k,l,u,v\in n$ with $u\neq v$ such that the sequence $\{d_t : t\in a\}$ converges for any $(u,v)$-split $(i,jk,l)$-doublecomb $a$. In this case, we restrict to a further first-move subtree $T_u= \{u^\frown t : t \in n^{<\omega}\}$. Fix a $(k,l)$-comb $c$ inside $T_v = \{v^\frown t : t \in n^{<\omega}\}$. For every $(i,j)$-comb $a$ inside $T_u$ it is possible to find  infinite $a'\subset a$ and $c'\subset c$ such that  $a'\cup c'$ is a $(u,v)$-split $(i,j,k,l)$-double comb. Therefore, we conclude that for all $(i,j)$-combs $a\subset T_u$, the sequence $\{d_t : t\in a\}$ converges to the same point, that we call $\mathbf{h}_\infty$. We claim that, indeed, every $(\tilde{\i},\tilde{\j})$-comb inside $T_u$ for any $\tilde{\i}\neq \tilde{\j}$, the set $\{d_t : t\in a\}$ converges to $\mathbf{h}_\infty$. In other words, we claim that $\mathbf{h}_{(z,\tilde{\i},\tilde{\j})} = \mathbf{h}_\infty$ whenever $\tilde{\i}\neq \tilde{\j}$. To see this, let $\{t_q : q<\omega\}$ be an $(\tilde{\i},\tilde{\j})$-comb over a branch $z$. For every $q<\omega$, let $\{t_{qp} : p<\omega\}$ be an $(i,j)$-comb such that $t_q < t_{qp}$ por all $p$. Since each sequence $\{t_{qp} : p<\omega\}$ converges to $\mathbf{h}_\infty$, we can apply Lemma~\ref{biseqblocks} and we obtain infinite $N,N_q\subset\omega$ such that the sets $\{\{d_{t_{qp}} : p\in N_q\} : q\in N\}$ converge to $\mathbf{h}_\infty$. It is possible to pick $p[q]\in N_q$ such that $\{t_{qp[q]} : q<\omega\}$ is an $(\tilde{\i},\tilde{\j})$-comb over $z$. This shows that  $\mathbf{h}_{(z,\tilde{\i},\tilde{\j})} = \mathbf{h}_\infty$. 

For $p,q\in n$ we have an equivalence relation $p\sim q$ induced by the equivalence relation defined on pairs before: $p\sim q$ if and only if $(p,p)\sim (q,q)$. Notice that if $\mathbf{h}_{(z,p,p)} = \mathbf{h}_\infty$ for some $z$ then $\mathbf{h}_{(z,p,p)} = \mathbf{h}_\infty$ for all $z$. This is because if $\mathbf{h}_{(z,p,p)} = \mathbf{h}_\infty$ then there is an $(i,j,p,p)$-double comb $a$ for which $\{d_t :t\in a\}$ is convergent and then $\{d_t : t\in a\}$ is convergent for every $(i,j,p,p)$-comb $a$. Let $\mathfrak{Q}$ be the set of the equivalence classes of $\sim$ on $n$ with the exception of the equivalence class of all $p$ such that $\mathbf{h}_{(z,p,p)} = \mathbf{h}_\infty$. We claim that the map $d_t\mapsto \mathbf{g}_t$ induces a homeomorphism between the closures. The proof is similar as in Case 1, since now we have a complete control on the convergence of all combs. If this map did not induce a continuous function between the closures, then we would have $\{d_t : t\in a\}$ and $\{d_t : t\in b\}$ that converge to the same point but $\{\mathbf{g}_t : t\in a\}$ and $\{\mathbf{g}_t : t\in b\}$ that converge to different points. We can suppose that $a$ is a $(p,q)$-comb over $x$ and $b$ is a $(p',q')$-comb over $y$, so that $\{d_t : t\in a\}$ converges to $\mathbf{h}_{(x,p,q)}$ and $\{d_t : t\in b\}$ converges to $\mathbf{h}_{(x,p',q')}$. But the discussion above implies that $\mathbf{h}_{(x,p,q)} = \mathbf{h}_{(x,p',q')}$ if and only if $\{\mathbf{g}_t : t\in a\}$ and $\{\mathbf{g}_t : t\in b\}$ converge to the same point according to Lemma~\ref{convergenceKinfty}. The same argument shows that the inverse mapping $\mathbf{h}_t \mapsto d_t$ also extends to a cotinuous function between their closures, so this continuous extension is a homeomorphism.
\end{proof}

Let $\mathfrak{Q}_2 = \{2\}$ be the trivial partition of $2=\{0,1\}$ into just one set, and for $n>2$, let $\mathfrak{Q}_n = \{\{0\},\{1\},\ldots,\{n-2\}\}$ be the partition of $n-1$ into singletons. In this way, $odeg(K_\infty(\mathfrak{Q}_n)) = n$.

\begin{thm}\label{mainthm}
Let $K$ be a separable Rosenthal compact space and $n$ a natural number. If $odeg(K)\geq n$ then  $K$ contains either a homeomorphic copy of $K_\infty(\mathfrak{Q}_n)$ or of $K_1(\mathfrak{P}_{f^\alpha})$ for some strong-dense-type $\alpha$ with $|\alpha|=n$.
\end{thm}

\begin{proof}
After applying Lemma~\ref{findseparatedK} and then Lemma~\ref{findstandardK}, we obtain $\{d_t : t\in n^{<\omega}\}$ such that
$$(\star\star)\ \ \ \overline{\{d_t : t^\frown i \leq z\}} \cap \overline{\{d_t : t^\frown j \leq z\}} = \emptyset$$
for all $z\in n^{\omega}$ and all $i<j<n$, and 
\begin{enumerate}
\item either $d_t\mapsto \mathbf{f}_t$ induces a homeomorphism of $\overline{\{d_t : t\in n^{<\omega}\}}$ with a space $K_1(\mathfrak{P})$,
\item or $d_t\mapsto \mathbf{g}_t$ induces a homeomorphism of  $\overline{\{d_t : t\in n^{<\omega}\}}$ with a space $K_\infty(\mathfrak{Q})$.
\end{enumerate}
In the first case, property $(\star\star)$ implies that each $(i,i)$ lies in a different piece of the partition $\mathfrak{P}$ for each $i\in n$, so in particular $|\mathfrak{P}|\geq n$. Then, Lemma~\ref{minimal} implies that $K$ contains a copy of $K_1(\mathfrak{P}_{f^\alpha})$ for some $\alpha$ with $|\alpha|=n$. In the second case, property $(\star\star)$ implies that for $i\neq j$, $i$ and $j$ cannot be in the same set from $\mathfrak{Q}$. This leaves only two possibilities, either $\mathfrak{Q} = \mathfrak{Q}_{n+1}$ consists of all singletons of $n$, or $\mathfrak{Q}$ consists of all singletons of $n$ except one, and we can suppose without loss of generality that this missed singleton is $\{n-1\}$. If $n>2$, in both cases we obtain that $\overline{\{d_t : t\in (n-1)^{<\omega}\}}$ is homeomorphic to $\mathfrak{Q}_n$ and we are done. If $n=2$, then in both cases we obtain that $$\overline{\{d_{(t_0,0,t_1,0,\ldots,t_k)} : (t_0,t_1,\ldots,t_k)\in 2^{<\omega}\}}$$
is homeomorphic to $K_\infty(\mathfrak{Q}_2)$. In all cases the homeomorphism is easily checked as we have done several times before using Lemmas~\ref{continuousondense} and~\ref{convergenceKinfty}, as it is clear where each kind of $(i,j)$-comb converges.

\end{proof}

\section{Some classical compact spaces}\label{sectionclass}

Notice that if $|\mathfrak{P}|=1$ then $K_1(\mathfrak{P})$ has open degree 1, so it is metrizable. In fact, in that case $K_1(\mathfrak{P})$ is a zero-dimensional compact metrizable space, hence homeomorphic to a subspace of the Cantor set.

\begin{lem}\label{containsthecantor}
$K_1(\mathfrak{P})$ contains a homeomorphic copy of the Cantor set if and only if there exist $i\neq j$ such that $(i,j)$ and $(j,i)$ live in the same piece of the partition.
\end{lem}

\begin{proof}
If $(i,j),(j,i)\in P$ live in the same piece $P$ of the partition, then  $\{\mathbf{f}_{(z,P)} : z\in \{i,j\}^\omega\}$ is homeomorphic to the Cantor set. This is easy to check using Proposition~\ref{convergenceK1}. Conversely, if $K_1(\mathfrak{P})$ contains a copy of the Cantor set, then it contains a copy of $K_1(\{0,1\}^2)$. Suppose that $\mathfrak{P} = \mathfrak{P}_{g^1}$ for some function $g^1$ and $g^0$ is the constant function equal to 0 on $\{0,1\}^2$. So we have $K_1(\mathfrak{P}_{g^0}) \subset K_1(\mathfrak{P}_{g^1})$. By Lemma~\ref{homeomorphiccopy}, we get that there exists $\sigma$ such that $\sigma\circ g^0 \prec g^1$. That is, there is a constant function on $\{0,1\}^2$ that is $\prec$-below $g^1$. Just using the definition, it is easy to check that this implies that there exists $(i,j)$ such that $g^1(i,j) = g^1(j,i)$. If otherwise, $g^1(i,j) \neq g^1(j,i)$ for all $i\neq j$, then that would imply that $g^0(0,1)\neq g^0(1,0)$.
\end{proof}

Consider the compact space  $S = [0,1]\times \{0,1\}$ endowed with the order topology induced by the lexicographical order. The split interval , also known as the double arrow space, is the space
$S' =S\setminus \{(0,0),(1,1)\}$ that we get after removing the two isolated points from $S$. It is not a difficult excercise to check that, for every perfect set $A\subset [0,1]$, the space $S(A) = A\times \{0,1\}\subset S$ contains a homeomorphic copy of $S'$. Indeed, every closed subset of $S$ without isolated points is order-isomorphic and homeomorphic to $S'$. We shall consider the space $S(2^\omega) = 2^\omega\times \{0,1\}$ where $2^\omega$ is identified with the Cantor set inside $[0,1]$, so that its order is the lexicographical order.

\begin{lem}\label{splitK1}
If $g:\{0,1\}^2\To \{0,1\}$ is such that $g(0,1)\neq g(1,0)$ and $\mathfrak{P}_g \neq \{\{(0,0),(1,0)\},\{(1,1),(0,1)\}\}$, then $K_1(\mathfrak{P}_g)$ is homeomorphic to a subspace of the split interval.
\end{lem}

\begin{proof}
We consider a homeomorphic embedding $\phi:K_1(\mathfrak{P}) \To S(2^\omega)$ defined as follows. For $x=(x_0,x_1,\ldots)\in 2^\omega$, let $$\phi(\mathbf{f}_{(x,g^{-1}(i))} )= ((x_0,0,1,x_1,0,1,x_2,0,1\ldots),i).$$ For $s=(s_0,\ldots,s_k)\in 2^{<\omega}$, we define

$\phi(\mathbf{f}_s) = \psi(s_0,0,1,s_1,0,1,\ldots,s_k,1,1,1,1,1,\ldots)$ if $g(0,0) = g(1,1) = g(0,1)$,

$\phi(\mathbf{f}_s) = \psi(s_0,0,1,s_1,0,1,\ldots,s_k,0,0,0,0,0,\ldots)$ if $g(0,0) = g(1,1) = g(1,0)$,

$\phi(\mathbf{f}_s) = \psi(s_0,0,1,s_1,0,1,\ldots,s_k,0,1,0,1,0,\ldots)$ if $g(0,0) = g(0,1)$ and $g(1,1) = (1,0)$.

Using Proposition~\ref{convergenceK1} and Lemma~\ref{continuousondense}, one easily checks that this is continuous one-to-one map.
\end{proof}

\begin{lem}\label{containsthesplit}
$K_1(\mathfrak{P})$ contains a homeomorphic copy of the split interval if and only if there exist $i\neq j$ such that $(i,j)$ and $(j,i)$ live in different pieces of the partition $\mathfrak{P}$.
\end{lem}

\begin{proof}
If $(i,j)\in P$ and $(j,i)\in P'$ live in different pieces of the partition, then  $\{\mathbf{f}_{(z,P)}, \mathbf{f}_{(z,P')} : z\in \{i,j\}^\omega\}$ is naturally homeomorphic to $S(2^\omega)$. Hence that set contains a copy of the split interval. The homeomorphism is easy to check using Proposition~\ref{convergenceK1} since in both spaces, the topology is determined by convergent sequences. If $K_1(\mathfrak{P})$ contains a copy of the split interval, then by Lemma~\ref{splitK1} it also contains copy of $K_1(\mathfrak{P}_{g^0})$ where $g^0:\{0,1\}^2\To \{0,1\}$ is $g^0(i,j) = i$.  Suppose that $\mathfrak{P}$ is of the form $\mathfrak{P} = \mathfrak{P}_{g^1}$, so by Lemma~\ref{homeomorphiccopy} there exists $\sigma$ such that $\sigma\circ g^0 \prec g^1$. But from the definition of $\prec$ it easily follows that if $g^1(i,j) = g^1(j,i)$ for all $i,j$, the same would happen for $g^0$.
\end{proof}

\begin{thm}\label{nonscattered}
Let $K$ be a Rosenthal compact space that is not scattered. Then $K$ contains either a homeomorphic copy of the Cantor set or a homeomorphic copy of the split interval.
\end{thm}

\begin{proof}
Since $K$ is not scattered, there exists a continuous surjection $\psi:K\To [0,1]$. For every $t=(t_1,\ldots,t_k)\in 2^{<\omega}$, let $z_t = \sum_{j=1}^{k} 2t_j 3^{-j}\in [0,1]$. In this way, $\overline{\{z_t : t\in 2^{<\omega}\}}$ is the Cantor set. For every $t\in 2^{<\omega}$, pick $d_t\in K$ such that $\psi(d_t) = z_t$. The key property of these points is that 
$$ (\clubsuit) \ \ \overline{\{d_t : s^\frown 0 \leq t\}} \cap \overline{\{d_t : s^\frown 1 \leq t\}} = \emptyset$$
for every $s\in 2^{<\omega}$. By Lemma~\ref{findstandardK}, we pass to a first-move subtree $T\approx 2^{<\omega}$ for which the natural bijection induces a homeomorphism between $\overline{\{d_t : t\in T\}}$ and the dense subset of a space of the form $K_1(\mathfrak{P})$ or $K_\infty(\mathfrak{Q})$ on the dyadic tree. However, property $(\clubsuit)$ eliminates the case of $K_\infty(\mathfrak{Q})$  because $\mathbf{g}_{\infty}$ would be in all those closures. By Lemmas~\ref{containsthecantor} and~\ref{containsthesplit}, the space $K_1(\mathfrak{P})$ contains a Cantor set if $(0,1)$ and $(1,0)$ lie in the same piece of the partition, and it contains a split interval if $(0,1)$ and $(1,0)$ lie in different pieces of the partition.
\end{proof}

Notice that the above proof shows something a litle bit stronger: if we have a continuous surjection from a Rosenthal compact space onto a non scattered space, then there is a closed subspace where the restriction is either a homeomorphism between Cantor sets or it looks like the canonical surjection $S(2^\omega)\To 2^\omega$.

A similar argument shows that if a Rosenthal compactum maps continuously onto the split interval, then it contains the split interval:

\begin{thm}\label{injectiveonsplit}
Let $K$ be a Rosenthal compact space and $\psi:K\To S$ be a continuous surjection from $K$ onto the split interval. Then, there exists a closed subset $Z\subset K$ homeomorphic to the split interval such that the restiction $\psi|_Z : Z\To L$ is one to one.
\end{thm}

\begin{proof}
Consider now $z^{-}_t = (z_t,0) \in S$ to be the left twin of the $z_t$ considered in the previous proof. We pick again $d_t$ such that $\psi(d_t) = z^{-}_t$, and we have the same key property $(\clubsuit)$ as above. This property also implies that $\psi\overline{\{d_t : t\in 2^{<\omega}\}}$ is uncountable. Using Lemma~\ref{findstandardK}, we suppose  that $\mathbf{f}_t\mapsto d_t$ induces a natural homeomorphism $$\phi: K_1(\mathfrak{P})\To \overline{\{d_t : t\in 2^{<\omega}\}}.$$ If $(0,1)$ and $(1,0)$ live in the same piece $P$ of the partition $\mathfrak{P}$,  then $\{\mathbf{f}_{(x,P)} : x\in 2^\omega\}$ would be homeomorphic to the Cantor set, but its image under $\psi\phi$ would be, by property $(\clubsuit)$, an uncountable subspace of the split interval. This is impossible because every metrizable closed subspace of the split interval is countable. So $(0,1)$ and $(1,0)$ lie in different pieces of the partition $P,P'\in \mathfrak{P}$. But then 
$$Z = \{\mathbf{f}_{(x,P)} : x\in 2^\omega\}\cup \{\mathbf{f}_{(x,P')} : x\in 2^\omega\}$$
is homeomorphic to $S(2^\omega)$ and $\psi$ is injective on $\phi Z$ by the choice of the elements $d_t$. 
\end{proof}

\section{Low degrees}\label{sectionlow}

In order togive the explicit list of the minimal separable Rosenthal compacta of degree $n$ we just have to enumerate all possible strong dense-types $\alpha = (A,B,C,D,E,\psi,\mathcal{P},\gamma)$ with $|\alpha|=n$, consider the corresponding $K_1(\mathfrak{P}_{f^\alpha})$ and adding $K_\infty(\mathfrak{Q}_n)$. The lists  of the strong types for $n=2$ and $n=3$ are found in \cite[Section 8]{stronggaps}, the dense-types being shorter sublists of them. In any case, this is an easy discussion of cases:

For $n=2$, we have, up to permutation, only two dense strong types $\alpha^0$ and $\alpha^1$:

$$
\begin{array}{|c|c|c|c|c|c|c|c|c|c|}
\hline
& m & A & B & C & D & E & \psi & \mathcal{P} & \gamma\\
 \hline
\alpha_2^0 & 2 &    &   & \{0,1\} &   &   &   & \{\{0,1\}\} &  \\
\hline
\alpha_2^1 & 2 & \{0\} &   & \{1\} &   &   &   & \{\{1\}\} &  \\
\hline
\end{array}
$$

The corresponding partitions of $2\times 2$ are $\mathfrak{P}_2^0 = \{\{ (0,0),(1,1),(1,0)\}, \{(0,1)\} \}$ and $\mathfrak{P}_2^1 = \{\{(0,0)\},\{(0,1),(1,0),(1,1)\}\}$. Thus, $K_1(\mathfrak{P}_2^0)$, $K_1(\mathfrak{P}_2^1)$ and $K_\infty(\mathfrak{Q}_2)$ form the basis of three elements for separable Rosenthal compact spaces.  By Lemmas~\ref{splitK1} and \ref{containsthesplit}, the space $K_1(\mathfrak{P}_2^0)$ both contains and is contained in the split interval. On the other hand, $\{\mathbf{f}_{(x,P)} : x\in 2^\omega, P\in \mathfrak{P}_2^1\}$ is homeomorphic to the so-called Alexandroff duplicate of the Cantor set, but this duplicate is non-separable while $K_1(\mathfrak{P}_2^1)$ is separable. These three minimal spaces are like the seven spaces of \cite{ADK}.

For $n=3$, the possible $\alpha$'s are:
$$
\begin{array}{|c|c|c|c|c|c|c|c|c|c|}
\hline
& m & A & B & C & D & E & \psi & \mathcal{P} & \gamma\\
 \hline
\alpha_3^0 & 2 & \{0\} &   & \{1,2\} &   &   &   & \{\{1,2\}\} &  \\
\hline
\alpha_3^1 & 3 & \{0\} &   & \{1,2\} &   &   &   & \{\{1\},\{2\}\} &  \\
\hline
\alpha_3^2 & 2 &  \{0,1\} & \{2\} &   &   &   & \equiv 2 &   &  \\
\hline
\end{array}
$$

\medskip
For $n=4$, the possible $\alpha$'s are (in this table, we omit some brackets):
$$
\begin{array}{|c|c|c|c|c|c|c|c|c|c|}
\hline
& m & A & B & C & D & E & \psi & \mathcal{P} & \gamma\\
 \hline
\alpha_4^0 & 3 &    &   &  0,1,2,3  &   &   &   &  \{0,1\},\{2,3\}   &  \\
\hline
\alpha_4^1 & 4 &  0  &   &  1,2,3  &   &   &   &  \{1\},\{2\},\{3\}  &  \\
\hline
\alpha_4^2 & 3 &  0  &   &  1,2,3  &   &   &   &  \{1\},\{2,3\}  &  \\
\hline
\alpha_4^3 & 3 &  0  &   &  &  1,2   &  3   &   &  & \equiv 3 \\
\hline
\alpha_4^4 & 2 &  0,1  &  2,3   &  &   &  & (0,1)\mapsto 2 &  &  \\
 &  &   &  &  &   &  &  (1,0)\mapsto 3  &  &  \\
\hline
\alpha_4^5 & 3 &  0,1  &  2   &  3   &   &  & \equiv 2 &  \{3\} &  \\
\hline
\alpha_4^6 & 3 &  0,1  &  2   &   &  3   &  & \equiv 2 &  & \equiv 2  \\
\hline
\alpha_4^7 & 3 &  0,1,2  &  3   &   &   &  & \equiv 3 &  &   \\
\hline
\end{array}
$$
\medskip

Thus there are three minimal separable Rosenthal compact spaces of open degree 2, four minimals of open degree 3 and eight minimals of open degree 4.

\section{Problems}\label{sectionprob}

The fact that Rosenthal compact spaces are sequentially compact can be rephrased by saying that a Rosenthal compact space is finite if and only if it does not contain homeomorphic copies of $K_\infty(\emptyset)$. One of the main results of \cite{1BC} can be reformulated in our language by saying that Rosenthal compact space is first countable if and only if it does not contain homeomorphic copies of $K_\infty(\mathfrak{Q}_2)$. In the same spirit, Corollary~\ref{nonscattered} can be restated by saying that
a Rosenthal compact space is scattered if and only if it does not contain copies of either $K_1(\mathfrak{P}_1)$ or $K_1(\mathfrak{P}_2)$, where $\mathfrak{P}_1 = \{ \{(0,0),(1,1),(0,1),(1,0)\} \}$ and $\mathfrak{P}_2 = \{\{(0,1)\},\{(1,0),(0,0),(1,1)\}\}$. All these results suggest a general problem: Given a fixed set of spaces of the form $K_1(\mathfrak{P})$ or $K_\infty(\mathfrak{Q})$, which is the class of Rosenthal compact spaces that do not contain any of them? Or from another point of view, which classes of Rosenthal compact spaces can be described as those that do not contain certain spaces  of the form $K_1(\mathfrak{P})$ or $K_\infty(\mathfrak{Q})$? For example, we do not know any characterisation of the class of Rosenthal compact spaces that do not contain the split interval.

\end{document}